\documentclass[12pt]{amsart} 
\usepackage{verbatim, latexsym, amssymb, amsmath,color, mathrsfs}
\usepackage{enumitem}
\usepackage{MnSymbol}
\usepackage{epsfig}
\usepackage{xcolor}

\usepackage{stmaryrd}
\usepackage{float}

\makeatletter
\renewcommand{\tocsection}[3]{%
  \indentlabel{\@ifnotempty{#2}{\bfseries\ignorespaces#1 #2\quad}}\bfseries#3}
\renewcommand{\tocsubsection}[3]{%
  \indentlabel{\@ifnotempty{#2}{\ignorespaces#1 #2\quad}}#3}

\newcommand\@dotsep{4.5}
\def\@tocline#1#2#3#4#5#6#7{\relax
  \ifnum #1>\c@tocdepth 
  \else
    \par \addpenalty\@secpenalty\addvspace{#2}%
    \begingroup \hyphenpenalty\@M
    \@ifempty{#4}{%
      \@tempdima\csname r@tocindent\number#1\endcsname\relax
    }{%
      \@tempdima#4\relax
    }%
    \parindent\z@ \leftskip#3\relax \advance\leftskip\@tempdima\relax
    \rightskip\@pnumwidth plus1em \parfillskip-\@pnumwidth
    #5\leavevmode\hskip-\@tempdima{#6}\nobreak
    \leaders\hbox{$\m@th\mkern \@dotsep mu\hbox{.}\mkern \@dotsep mu$}\hfill
    \nobreak
    \hbox to\@pnumwidth{\@tocpagenum{\ifnum#1=1\bfseries\fi#7}}\par
    \nobreak
    \endgroup
  \fi}
\AtBeginDocument{%
\expandafter\renewcommand\csname r@tocindent0\endcsname{0pt}
}
\def\l@subsection{\@tocline{2}{0pt}{2.5pc}{5pc}{}}
\makeatother

\usepackage{geometry}
\usepackage{hyperref}

\DeclareMathOperator{\Id}{Id}
\DeclareMathOperator{\diam}{diam}

\DeclareMathOperator{\Vol}{Vol}
\DeclareMathOperator{\Area}{Area}
\DeclareMathOperator{\Length}{Length}

\DeclareMathOperator{\Jac}{Jac}

\DeclareMathOperator{\dvol}{dvol}

\newtheorem{thm}{Theorem}[section]
\newtheorem{prop}[thm]{Proposition}
\newtheorem{lem}[thm]{Lemma}

\newtheorem{conj}[thm]{Conjecture}
\newtheorem{rmk}[thm]{Remark}

\theoremstyle{defn}

\newtheorem{example}[thm]{Example}

\begin{document}
\title[Stability of Euclidean 3-space for the positive mass theorem]
{Stability of Euclidean 3-space for the positive mass theorem}

\author{Conghan Dong}
\address{Mathematics Department, Stony Brook University, NY 11794, United States}
\email{conghan.dong@stonybrook.edu}

\author{Antoine Song}
\address{California Institute of Technology\\ 177 Linde Hall, \#1200 E. California Blvd., Pasadena, CA 91125}
\email{aysong@caltech.edu}

\begin{abstract} 
	We show that the Euclidean 3-space $\mathbb{R}^3$ is stable for the Positive Mass Theorem in the following sense. Let $(M_i,g_i)$ be a sequence of complete asymptotically flat $3$-manifolds with nonnegative scalar curvature and suppose that the ADM mass $m(g_i)$ of one end of $M_i$ converges to $0$. Then for all  $i$, there is a subset $Z_i$ in $M_i$ such that 
$M_i\setminus Z_i$ contains the given end, the area of the boundary $\partial Z_i$ converges to zero, and $(M_i\setminus Z_i,g_i)$ converges to $\mathbb{R}^3$ in the pointed measured Gromov-Hausdorff topology for any choice of basepoints.
This confirms a conjecture of G. Huisken and T. Ilmanen. Additionally, we find an almost quadratic upper bound for the area of $\partial Z_i$ in terms of  $m(g_i)$. {As an application of the main result, we also prove R. Bartnik's strict positivity conjecture.}

\end{abstract}

\maketitle

\section{Introduction}
According to the Riemannian Positive Mass Theorem in dimension 3, the ADM mass of an end of a complete  asymptotically flat Riemannian 3-manifold $(M,g)$ with nonnegative scalar curvature is nonnegative. Furthermore, the mass is zero if and only if $(M,g)$ is isometric to  Euclidean  3-space $(\mathbb{R}^3,g_{\mathrm{Eucl}})$. Originally proved by Schoen-Yau \cite{SY79}, this theorem has since been proven by numerous new methods \cite{Witten81,HI01,Li18,BKKS22,AMO21,Miao22}, and many fundamental extensions were discovered,  including the Penrose inequality \cite{HI01,Bray01,BL09,AMMO22},  the spacetime Positive Mass Theorem \cite{SY81,Witten81,Bartnik86,Eichmair13,EHLS15} and other generalizations \cite{SY17,LUY21} etc. 

Recently there has been growing interest in establishing a stability result for the Positive Mass Theorem. This problem falls within the larger effort  to understand the (non-)stability of classical rigidity theorems for metrics with lower scalar curvature bounds \cite{Sormani23,Gromov19}. When it comes to the Positive Mass Theorem, one difficulty comes from the fact that asymptotically flat manifolds with nonnegative scalar curvature and small masses can be far from Euclidean  3-space with respect to standard topologies like the Gromov-Hausdorff topology, see for instance \cite{HI01,LS14}. To overcome this, researchers have turned to exploring newer topologies, notably the intrinsic flat topology, see \cite{Sormani23,Sormani16} for a survey on those questions.

The earliest conjecture about the stability of the Positive Mass Theorem was formulated by Huisken-Ilmanen \cite[Section 9]{HI01}:

\begin{conj}
Suppose $(M_i,g_i)$ is a sequence of asymptotically flat $3$-manifolds with nonnegative scalar curvature and suppose that the ADM mass $m(g_i)$ of an end converges to $0$. Then, there is a subset $Z_i\subset M_i$ such that 
$M_i\setminus Z_i$ contains the given end, the area of $\partial Z_i$ converges to $0$, and $(M_i\setminus Z_i,g_i)$ converges to Euclidean 3-space in the Gromov-Hausdorff topology for any choice of basepoints.
\end{conj}

Because of the Penrose inequality, the following has also been conjectured by Ilmanen \cite{Bray22}:
\begin{conj}
The subsets $Z_i$ can be chosen so that 
$$\Area(\partial Z_i) \leq 16\pi m(g_i)^2. $$
\end{conj}

The purpose of this paper is to prove the stability conjecture of Huisken-Ilmanen. Our results go slightly beyond the conjecture, as we show that convergence occurs also in the sense of measures. Although we are unable to settle the second conjecture, we make progress by providing a bound of the surface area which is almost quadratic in the mass. The main statement is the following:

\begin{thm}\label{main thm}
	Let $(M_i,g_i)$ be a sequence of {complete} asymptotically flat $3$-manifolds with nonnegative scalar curvature. Suppose that the ADM mass $m(g_i)$ of one end $\mathscr{E}_i$ of $M_i$ converges to $0$. Then for all $i$, there is an open domain $Z_i$ in $M_i$ with smooth compact boundary $\partial Z_i$, such that 
	\begin{enumerate}
	\item $M_i\setminus Z_i$ contains the given end $\mathscr{E}_i$,
	\item
	the area of $\partial Z_i$ converges to $0$, 
	\item  for any choice of basepoint $p_i \in M_i\setminus Z_i$,
	$$(M_i\setminus Z_i, \hat{d}_{g_i}, p_i)\to (\mathbb{R}^3, d_{\mathrm{Eucl}}, 0)$$ 
	in the pointed measured Gromov-Hausdorff topology, where $\hat{d}_{g_i}$ is the length metric on $M_i\setminus Z_i$ induced by $g_i$.
	\end{enumerate}

	Moreover, if $m(g_i)>0$ for all $i$, the region $Z_i$ can be chosen so that the area of $\partial Z_i$ is almost bounded quadratically by the mass in the following sense: for any positive continuous function $\xi :(0, \infty)\to (0,\infty)$ with  
	$$\lim_{x\to 0^+}{\xi (x)} =0,$$
	for all large $i$, we {can choose $Z_i$ depending on $\xi $ such that} $$
\Area(\partial Z_i) \leq \frac{m(g_i)^2}{\xi (m(g_i) )}.
$$   

\end{thm}

 {To clarify, the length metric $\hat{d}_{g_i}$ in this statement measures the distance between two points $x,y\in M_i\setminus Z_i$ as the infimum of the Riemannian $g_i$-lengths of paths contained \emph{inside} $M_i\setminus Z_i$ and joining $x$ to $y$.} Our proof will actually show that $(M_i\setminus Z_i, g_i, p_i)$ becomes close to $(\mathbb{R}^3, d_{\mathrm{Eucl}}, 0)$ in a $C^0$ sense (although this does not imply Gromov-Hausdorff convergence).

{As pointed out to us by G. Huisken, the proof of Theorem \ref{main thm} can be applied to settle an open question about the rigidity of the Bartnik capacity $c_B(\Omega)$ (or Bartnik's quasi-local mass) of an admissible open Riemannian $3$-manifold $\Omega$.
In \cite[Positivity Property 9.1]{HI01}, G. Huisken and T. Ilmanen found that either $c_B(\Omega)>0$, or $c_B(\Omega)=0$ and in that case $\Omega$ is a locally flat Riemannian manifold. R. Bartnik \cite{Bartnik89} conjectured that if $c_B(\Omega)=0$, then $\Omega$ is actually isometrically embedded inside Euclidean $3$-space $\mathbb{R}^3$. This is indeed true, as explained in Section \ref{cb}:
\begin{thm} 
If $\Omega$ is admissible, then $c_B(\Omega)>0$, unless there is a Riemannian isometric embedding of $\Omega$ into the Euclidean 3-space $\mathbb{R}^3$. 
\end{thm}
On the other hand, a construction due to M. Anderson and J. Jauregui \cite[Theorem 1.2]{AJ19} suggests that  the metric closure of $\Omega$ might not always be isometrically embedded inside Euclidean $3$-space $\mathbb{R}^3$ (see Remark \ref{aj} for more comments). 
}
\vspace{2em}

The stability problem for the Positive Mass Theorem has been the subject of several studies. Many of those earlier works focused on proving stability results using the intrinsic flat topology of Sormani-Wenger \cite{SW11},  following the conjecture formulated by Lee-Sormani in \cite{LS14}. The case of spherically symmetric asymptotically flat manifolds was addressed  by Lee-Sormani \cite{LS14}, and was extended in \cite{BKS21}. The graphical setting was settled by papers of Huang, Lee, Sormani, Allen, Perales \cite{HL15,HLS17,HLP22,AP20}. Additionally, Sobolev bounds were obtained in \cite{Allen19}, and there were prior results on establishing $L^2$ curvature bounds and proving stability outside of a compact set \cite{BF99,FK01,Corvino05,Finster07,Lee09}. 
More recently, Kazaras-Khuri-Lee \cite{KKL21} were able to prove the stability conjecture under Ricci curvature lower bounds and a uniform asymptotic flatness assumption. See also \cite{ABK22} for stability under integral Ricci curvature bounds, isoperimetric bounds and a uniform asymptotic flatness assumption. The first named author \cite{Dong22} made progress on the original version of the stability conjecture, without additional curvature assumptions but still under a uniform asymptotic flatness assumption. Other recent works on stability for scalar curvature include results for tori with almost nonnegative scalar curvature \cite{LNN20,Allen21,CL22}.

Our main result, Theorem \ref{main thm}, is a stability theorem up to negligible subsets.
In particular, for the Positive Mass theorem, $\mathbb{R}^3$ is ``codimension 2 stable'' in the measured Gromov-Hausdorff topology, see \cite[Remark 3.9]{Song23}. 
This might be optimal in view of recent work of Kazaras-Xu \cite{KX23}. 
In \cite{Song23}, it was shown that hyperbolic manifolds of dimension $n\geq 3$ are stable for the entropy inequality, after removing negligible subsets similarly to Theorem \ref{main thm}.
{Besides, based on the techniques of this paper, the first named author recently obtained a stability result for the Riemannian Penrose inequality \cite{Dong24} analogous to Theorem \ref{main thm}.

\subsection*{Outline of proof}
Similarly to \cite{KKL21,Dong22}, our proof builds on the recent beautiful new proof of the Positive Mass Theorem by Bray-Kazaras-Khuri-Stern \cite{BKKS22}. There, the authors employ level sets of harmonic maps \cite{Stern22}: contrarily to methods in other proofs of the theorem, level sets of harmonic maps sweepout the whole manifold and are sensitive to the global geometry. That heuristically explains why \cite{BKKS22} seems to be a good starting point when attempting to prove stability. Perhaps surprisingly, aside from \cite{BKKS22}, all the arguments we use are completely elementary. We emphasize that unlike \cite{KKL21,Dong22}, we do not assume additional curvature bounds or uniform asymptotic flatness.

The proof begins with the result of Bray-Kazaras-Khuri-Stern \cite{BKKS22}, which bounds the mass $m(g)$ of an end of a complete asymptotically flat manifold $(M,g)$ with nonnegative scalar curvature as follows: let ${M_{ext}}$ be an exterior region of $(M,g)$ containing that end, then
\begin{equation} \label{bkks0}
m(g) \geq \frac{1}{16} \int_{M_{ext}} \frac{|\nabla^2 u|^2}{|\nabla u|} \dvol_g
\end{equation}
where $ u:{M_{ext}}\to \mathbb{R}$ is any harmonic function ``asymptotic to one of the asymptotically flat coordinate functions on $M_{ext}$'', see \cite{BKKS22}.

Let $\mathbf{u}:=(u^1,u^2,u^3) :M_{ext} \to \mathbb{R}^3 $, where $u^j$ ($j=1,2,3$) is the harmonic function asymptotic to the asymptotically flat coordinate function $x^j$ on $M_{ext}$ (we fix a coordinate chart at infinity). Assuming that the mass $m(g)$ is positive and close to $0$, we consider as in \cite{Dong22} the subset $\Omega$ of $M_{ext}$ where the differential of $\mathbf{u}$ is $\epsilon_0$-close to the identity map, for some small positive $\epsilon_0$. More concretely, define $F: (M_{ext},g)\to \mathbb{R}$ by 
$$
F(x):= \sum_{j,k=1}^3 \left( g(\nabla u^j, \nabla u^k) - \delta _{jk} \right) ^2.
$$ 
Then set $\Omega:=F^{-1}[0,\epsilon_0]$. On such a set, $\mathbf{u}$ is close to being a Riemannian isometry onto its image in $\mathbb{R}^3$. It is a local diffeomorphism, but a priori, the multiplicities of this map are not easily controlled. The proof is divided into two steps.

In the first step, we show that one regular level set $S$ of  $F$ has small area, bounded by $m(g)^2$ times a constant depending on $\epsilon_0$. Inequality (\ref{bkks0}) bounds the $L^2$ norm of the gradient of $F$. The crucial point is then to observe that the classical capacity-volume inequality of Poincar\'{e}-Faber-Szeg\"{o} generalizes to our setting, and enables us to find a region between two level sets of $F$ with small volume. Then one can apply the coarea formula to find the small area level set $S$. 

Consider the connected component $E$ of $M_{ext}\setminus S$ containing the end. An issue is that even though the area of $\partial E$ is small, the component $E$ is in general not close to Euclidean 3-space for the induced length structure on $E$.  The second step of the proof consists of modifying the subset $E$ to another subset $E''$ containing the end, so that $\partial E''$ still has small area but moreover $E''$ is close to Euclidean 3-space in the pointed Gromov-Hausdorff topology with respect to its own length metric. This result is quite general, and is shown  by dividing  the full space into small cube-like  regions and by a repeated use of the coarea formula.

\subsection*{Acknowledgments:}
We would like to thank Gerhard Huisken for suggesting an application of the main theorem to the Bartnik capacity, and Marcus Khuri, Hubert Bray, Christos Mantoulidis and Hyun Chul Jang for helpful discussions. 
 The writing of this paper was substantially improved thanks to suggestions of the referees.

C.D. would like to thank his advisor Xiuxiong Chen for his encouragements. Part of the revised version was supported by the NSF grant DMS-1928930, while C.D. was in residence at the Simons Laufer Mathematical Sciences Institute (formerly MSRI) in Berkeley, California, during the Fall 2024 semester.

A.S. was partially supported by the NSF grant DMS-2104254. This research was conducted during the period A.S. served as a Clay Research Fellow.

\section{Preliminaries}

\subsection{Notations} \label{notations}
We will use $C$ to denote a universal positive constant (which may be different from line to line); $\Psi (t), \Psi (t |a, b, \ldots)$ denote small constants depending on $a, b, \ldots$ and satisfying $$\lim_{t \to 0}\Psi (t)=0, \ \lim_{t\to 0}\Psi (t |a,b, \ldots)=0,$$ for each fixed $a, b, \ldots$

We denote the Euclidean metric, the induced length function and the induced distance by $g_{\mathrm{Eucl}}, L_{\mathrm{Eucl}}, d_{\mathrm{Eucl}}$ respectively; and the geodesic line segment and distance between $x,y \in \mathbb{R}^3$ with respect to $g_{\mathrm{Eucl}}$ by $[xy], |xy|$. So $|xy|= L_{\mathrm{Eucl}}([xy])$. 

For a general Riemannian manifold $(M,g)$ and any $p \in M$, the geodesic ball with center $p$ and radius $r$ is denoted by $B_g(p, r)$ or $B(p, r)$ if the underlying metric is clear. Given a Riemannian metric, for a surface $\Sigma$ and a domain $\Omega$, $\Area(\Sigma)$ is the area of $\Sigma$, $\Vol(\Omega)$ is the volume of $\Omega $ with respect to the metric.

\subsection{Asymptotically flat $3$-manifolds}
A smooth orientable connected complete Riemannian $3$-manifold $(M, g)$ is called asymptotically flat if there exists a compact subset $K \subset M$ such that $M\setminus K= \bigsqcup_{k=1}^{N} M_{\mathrm{end}}^k$ consists of finite pairwise disjoint ends, and for each $1\leq k \leq N$, there exist $B>0, \sigma > \frac{1}{2}$, and a $C^\infty$-diffeomorphism $\Phi_k : M_{\mathrm{end}}^k \to  \mathbb{R}^3 \setminus B_{\mathrm{Eucl}}(0, 1)$ such that under this identification, 
	$$
	| \partial ^{l}(g_{ij}- \delta _{ij})(x)| \leq B|x|^{-\sigma - |l|},
	$$ for all multi-indices $|l|=0,1,2$ and any $x \in \mathbb{R}^3\setminus B_{\mathrm{Eucl}}(0,1)$, where $B_{\mathrm{Eucl}}(0,1)$ is the standard Euclidean ball with center $0$ and radius $1$. Furthermore, we always assume the scalar curvature $R_g$ is integrable over $(M,g)$. For each given end, its ADM mass from general relativity is then well-defined (see \cite{ADM61, Bartnik86}) and given by
	$$m(g):= \lim_{r\to \infty} \frac{1}{16\pi} \int_{S_r} \sum_{i,j} (g_{ij,i}-g_{ii,j}) v^j dA$$
	where $v$ is the unit outer normal to the coordinate sphere $S_r$ of radius $|x|=r$ in the given end, and $dA$ is its area element.

Let $(M, g)$ be a complete asymptotically flat  $3$-manifold. 
To each end of $M$, there is an associated ``exterior region'' $M_{ext}$, which contains that end, is diffeomorphic to $\mathbb{R}^3$ minus finitely many disjoint balls, and has minimal boundary. 
Given $(M_{ext}, g)$ associated to one end, let $\{x^j\}_{j=1}^{3}$ denote the asymptotically flat coordinate system of the end. Then there exist functions $u^j \in C^\infty(M_{ext}), j \in \{1,2,3\} ,$ satisfying the following harmonic equations with Neumann boundary if $\partial M_{ext} \neq \emptyset$, and asymptotically linear conditions:
$$
\Delta _g u^j =0 \text{ in } M_{ext},\ \frac{\partial u^j}{\partial \vec{n}}=0 \text{ on } \partial M_{ext},\ |u^j - x^j| = o(|x|^{1- \sigma  }) \text{ as } |x|\to \infty,
$$ 
where $\vec{n}$ is the normal vector of $\partial M_{ext}$, and $\sigma  >\frac{1}{2}$ is the order of the asymptotic flatness. We say that $\{u^j\}_{j=1}^3$ are the harmonic functions asymptotic to the asymptotically flat coordinate functions $\{x^j\}_{j=1}^3$ respectively. See \cite{BKKS22, KKL21} for more details.

Denote by $\mathbf{u}$ the resulting harmonic map
$$
\mathbf{u}:= (u^1, u^2, u^3): M_{ext} \to \mathbb{R}^3.
$$

The following mass inequality was proved in \cite{BKKS22}.
\begin{prop}[Theorem 1.2 in \cite{BKKS22}]\label{integral-ineq}
	Let $(M_{ext}, g)$ be an exterior region of an asymptotically flat Riemannian $3$-manifold $(M,g)$ with mass $m(g)$. Let $u$ be a harmonic function on $(M_{ext}, g)$ asymptotic to one of the asymptotically flat coordinate functions of the associated end. Then 
	\begin{equation}\label{mass-ineq}
	m(g) \geq \frac{1}{16\pi}\int_{{M_{ext}}}  \left( \frac{|\nabla ^2 u|^2}{| \nabla u|} + { R_g |\nabla u| } \right)  \dvol_g,
\end{equation}
{ where the integral is taken over the regular set of $u$.}\end{prop}

\subsection{Pointed measured Gromov-Hausdorff convergence}
In this subsection, we recall some definitions for the pointed measured Gromov-Hausdorff topology \cite[Definition 3.24]{GMS15}, also called pointed Gromov-Hausdorff-Prokhorov topology.

Assume $(X,d_X,x), (Y,d_Y, y)$ are two pointed metric spaces.
The pointed Gromov-Hausdorff (or pGH-) distance is defined in the following way. A pointed map $f:(X,d_X,x)\to (Y, d_Y,y)$ is called an $\varepsilon $-pointed Gromov-Hausdorff approximation (or $\varepsilon $-pGH approximation) if it satisfies the following conditions:
\begin{itemize}
	\item[(1)] $f(x)=y$;
	\item [(2)] $B(y, \frac{1}{\varepsilon }) \subset B_\varepsilon (f(B(x, \frac{1}{\varepsilon })) )$;
	\item[(3)] $|d_X(x_1, x_2) - d_Y(f(x_1), f(x_2) )|<\varepsilon $ for all $x_1, x_2 \in B(x, \frac{1}{\varepsilon })$.
\end{itemize}
The pGH-distance is defined by
$$
d_{pGH}( (X,d_X, x), (Y,d_Y,y) ) := \inf \{\varepsilon>0 : \exists\ \varepsilon\text{-pGH approximation } f:(X,d_X,x)\to (Y,d_Y,y)\} .
$$ 

We say that a sequence of pointed metric spaces $(X_i, d_i, p_i)$ converges to a pointed metric space $(X, d, p)$ in the pointed Gromov-Hausdorff topology, if $$d_{pGH}( (X_i, d_i, p_i), (X, d, p) ) \to 0.$$

If $(X_i, d_i)$ are length metric spaces, i.e.  for any two points $x, y \in X_i$, $$
d_i(x,y)= \inf \{L_{d_i}(\gamma ): \gamma \text{ is a rectifiable curve connecting }x, y\}, 
$$ where $L_{d_i}(\gamma )$ is the length of $\gamma $ induced by the metric $d_i$, then equivalently, $$d_{pGH}( (X_i, d_i, p_i), (X,d,p) )\to 0$$ if and only if for all $D>0$, $$d_{pGH}( (B(p_i, D), d_i), (B(p, D), d) ) \to 0,$$ where $B(p_i, D)$ are the geodesic balls of metric $d_i$.

A pointed metric measure space is a structure $(X,d_X, \mu, x )$ where $(X,d_X)$ is a complete separable metric space, $\mu $ is a Radon measure on $X$ and $x \in \mathrm{supp} (\mu)$. 

We say that a sequence of pointed metric measure length spaces $(X_i, d_i, \mu _i, p_i)$ converges to a pointed metric measure length space $(X, d, \mu , p)$ in the pointed measured Gromov-Hausdorff (or pointed Gromov-Hausdorff-Prokhorov) topology, if for any $\varepsilon >0, D>0$, there exists $N(\varepsilon ,D) \in \mathbb{Z}_+$ such that for all $i \geq N(\varepsilon , D)$, there exists a Borel $\varepsilon $-pGH approximation 
$$f_i^{D, \varepsilon }: (B(p_i, D), d_i, p_i) \to (B(p, D+\varepsilon ), d, p) $$
such that
$$
(f_i^{D, \varepsilon })_{\sharp}(\mu _i|_{B(p_i, D)}) \text{ weakly converges to } \mu |_{B(p, D)}  \text{ as } i\to \infty, \text{ for }a.e. D>0.
$$ 

In the standard case when $X_i$ is an $n$-dimensional manifold, without extra explanations, we will always consider $(X_i, d_i, p_i)$ as a pointed metric measure space equipped with the $n$-dimensional Hausdorff measure $\mathcal{H}^n_{d_i}$ induced by $d_i$.

\section{Regular region with small area boundary}
In this section, we consider a complete asymptotically flat $3$-manifold $(M, g)$ with nonnegative scalar curvature, and with a given end having small mass $m(g)$. Let $M_{ext}$ be the exterior region associated to that end.
 We will find an unbounded domain in $M_{ext}$ containing the end such that its compact boundary has small area depending on $m(g)$. We always assume $0<m(g) \ll1$.

Consider the harmonic map $\mathbf{u}=(u^1, u^2, u^3): M_{ext}\to \mathbb{R}^3$ associated to the given end as in previous section, where for each $j \in \{1,2,3\} $, $u^j$ is a harmonic function with Neumann boundary condition and asymptotic to an asymptotically flat coordinate function $x^j$ of the end.

Define the $C^{\infty}$-function $F: (M_{ext},g)\to \mathbb{R}$ by 
$$
F(x):= \sum_{j,k=1}^3 \left( g(\nabla u^j, \nabla u^k) - \delta _{jk} \right) ^2.
$$ 
Fix a small number $0 < \epsilon \ll 1$. We use the following notations: 
$$
\forall t \in [0, 6\epsilon ],\ S_t:= F^{-1}(t),
$$ 
$$
\forall 0 \leq a<b \leq 6\epsilon ,\ \Omega _{a,b}:= F^{-1}([a,b]).
$$ 
Notice that for any $t \in (0, 6 \epsilon ]$, $S_t$ is compact, $S_t \cap \partial M_{ext}= \emptyset$ since $\epsilon$ is small, and the complement of $\Omega _{0, t}$ is compact in $M_{ext}$. By Sard's theorem, we will always consider {regular values} $t \in (0, 6\epsilon ]$ outside of the measure zero set of critical values of $F$. For {regular values} $0 < a<b <  6 \epsilon $, we have $\partial \Omega _{a,b}= S_a \cup S_b$. 

We will always consider the restricted map $\mathbf{u}: \Omega _{0, 6\epsilon }\to \mathbb{R}^3$. We choose $\epsilon\ll 1 $ such that for any $x \in \Omega _{0, 6\epsilon }$, the Jacobian matrix $\Jac \mathbf{u}$ of $\mathbf{u}$ satisfies (with an abuse of notations):
\begin{equation} \label{abu}
|\Jac \mathbf{u}(x) - \Id| \leq \epsilon' ,
\end{equation} 
for 
$$\epsilon' := 100 \sqrt{\epsilon}\ll 1,$$ so that in particular  $\mathbf{u}$ is a local diffeomorphism. What we mean by (\ref{abu}) is that there exist orthonormal bases of $T_x M$ and $\mathbb{R}^3$ such that with respect to these bases, the Jacobian matrix $\Jac \mathbf{u}(x)$ is $\epsilon'$-close to the identity map.

We use the notation $C$ as explained in \ref{notations}. 

\begin{lem}\label{gradient F}
	$$
	\int_{\Omega _{0, 6\epsilon }}|\nabla F|^2 \leq C m(g).
	$$ 
\end{lem}
\begin{proof}
	By definitions, we have
\begin{equation}\label{gradient u^j}
 { \sqrt{1 - \sqrt{6 \epsilon } } \leq } 	|\nabla u^j|(x) \leq \sqrt{1+ \sqrt{6\epsilon } },\ \forall x \in \Omega _{0, 6\epsilon },\ \forall j \in \{1,2,3\} .
\end{equation}
	We readily obtain that for all $x \in \Omega _{0, 6\epsilon }$,
	\begin{align}
		\begin{split}
			|\nabla F|(x)& \leq \sum_{j,k=1}^3 2| g(\nabla u^j, \nabla u^k)- \delta _{jk}|\cdot (|\nabla u^j| |\nabla ^2 u^k| + | \nabla ^2 u^j| |\nabla u^k| )\\
			&\leq C\sum_{j=1}^3 |\nabla ^2 u^j|.
		\end{split}
	\end{align}
	So by inequalities (\ref{gradient u^j}) and (\ref{mass-ineq}), we have
	\begin{align}
		\begin{split}
			\int_{\Omega _{0, 6\epsilon }}|\nabla F| ^2 &\leq C \sum_{j=1}^3\int_{\Omega _{0, 6\epsilon }} |\nabla ^2 u^j|^2 \\
							    &\leq C \sum_{j=1}^3\int_{\Omega _{0, 6\epsilon }} \frac{|\nabla ^2 u^j|^2}{|\nabla u^j|}\\
							    & \leq C m(g).
		\end{split}
	\end{align}

\end{proof}

In the context of General Relativity, the notion of capacity of a set has been studied in \cite{Bray01,BM08,Jauregui20,Miao22} (see also references therein). We will use capacity in a different way. 
Recall that the classical Poincar\'e-Faber-Szeg\"o inequality relates the capacity of a set in Euclidean space to its volume, see \cite{PS51,Jauregui12}.
In the following lemma, which is a key step in this section, we prove a Poincar\'e-Faber-Szeg\"o type inequality. This will be used to find a smooth level set of $F$ with small area. 

\begin{lem}\label{PFS-ineq}
		If $\inf_{s \in (0, 5\epsilon )}\Area(S_s) >0$, where the infimum is taken over all  regular values, then there exists 
		$s_0 \in (0, 5\epsilon )$ such that 
	$$
	\Vol(\Omega _{s_0 , s_0+\epsilon } \cap \{|\nabla F| \neq 0\} ) \leq C \left( \frac{m(g)}{\epsilon ^2}\right) ^{3}.
	$$ 
\end{lem}

\begin{proof}
	Since $\inf_{s \in (0, 5\epsilon )}\Area(S_s)>0$, we can find a regular $s_0  \in (0,5\epsilon )$ such that $S_{s_0 }$ is a smooth surface satisfying
	$$
	\Area(S_{s_0 }) \leq 2\inf_{s \in (0, 5\epsilon )}\Area(S_s) .
	$$ 
For the reader's convenience, we follow the presentation of the note \cite{Jauregui12} when we can. 

	\textbf{Case 1)}: $s_0  \in [ 3\epsilon, 5\epsilon )  $.

For any regular value $s \in ( \epsilon  , s_0  )$, define the function $\chi : \mathbb{R}^3 \to \mathbb{R}$ by 
$$
\chi (x):= \mathcal{H}^{0}( \mathbf{u}^{-1}(x)  \cap \Omega _{s, s_0 }).
$$
By the isoperimetric inequality \cite[Theorem 5.10 (i)]{EG15}, 
$$
\| \chi \|_{L^{\frac{3}{2}}(\mathbb{R}^3)} \leq C_{\mathrm{isop}} \| D \chi \|(\mathbb{R}^3).
$$ 
Since $\mathbf{u}$ is a local diffeomorphism and $|\Jac \mathbf{u}-\Id| \leq \epsilon' $ by (\ref{abu}), this means that
\begin{equation}\label{s_0-s-isop}
\Vol(\Omega _{s, s_0 })^{\frac{2}{3}} \leq 2C_{\mathrm{isop}}( \Area(S_s) + \Area(S_{s_0 })).
\end{equation} 
By the definition of $s_0 $, we have
\begin{equation}\label{3C-isop}
	\Vol(\Omega _{s, s_0 })^{\frac{2}{3}} \leq 6C_{\mathrm{isop}} \Area(S_s) .
\end{equation}

By the coarea formula,
$$
\int_{\Omega _{\epsilon , s_0  }}|\nabla F|^2 \dvol_g = \int_{0}^{s_0- \epsilon } \int_{S_{s_0-t}}|\nabla F| dA_gdt.
$$ 
For $t \in (0, s_0-\epsilon )$, define
$$
V(t):= \Vol( \Omega _{s_0-t , s_0}),\ S(t):= \Area(S_{s_0-t}).
$$ 
Since $\inf_{s \in (0, 5\epsilon )}\Area(S_s) >0$, we have $S_{s_0-t} \neq \emptyset$. By Sard's theorem, we know that except possibly for a measure zero set in $(0, s_0- \epsilon )$,
$|\nabla F| \neq 0$ over $S_{s_0-t}$. 
Then
$$
W(t) := \int_{0 }^{ t } \int_{S_{s_0-s}}\frac{1}{|\nabla F|}dA_g ds
$$ 
is a strictly increasing continuous function, where the integral is taken over regular values of $F$,
and
$$
W'(t)= \int_{S_{s_0-t}}\frac{1}{|\nabla F|}dA_g>0
$$
is well-defined for a.e. $t \in (0, s_0-\epsilon )$. Notice that for $t \in (0, s_0-\epsilon )$,
$$0<W(t) = \Vol(\Omega _{s_0-t, s_0} \cap \{|\nabla F| \neq 0\} ) \leq V(t).$$

Since for a.e. $t \in (0, s_0-\epsilon )$,
$$
S(t)^2 \leq \int_{S_{s_0-t}}|\nabla F| dA_g \cdot \int_{S_{s_0-t}}\frac{1}{|\nabla F|}dA_g,
$$ 
we have
$$
\int_{0}^{s_0-\epsilon }\frac{S(t)^2}{W'(t)} dt \leq \int_{\Omega _{\epsilon , s_0}}|\nabla F|^2 \dvol_g.
$$ 
By the isoperimetric inequality (\ref{3C-isop}) obtained above,
$$
W(t)^{\frac{2}{3}} \leq V(t)^{\frac{2}{3}} \leq 6C_{\mathrm{isop}}S(t),
$$ 
so we have
\begin{equation} \label{utile}
\int_{0 }^{s_0-\epsilon }\frac{W(t)^{\frac{4}{3}}}{W'(t)}dt \leq 36C_{\mathrm{isop}}^2\int_{\Omega _{\epsilon , s_0}}|\nabla F|^2 \dvol_g.
\end{equation}

For any $t \in (0, s_0-\epsilon) $, define $R(t)$ by
$$
\omega _3 R(t)^{3} = W(t),
$$
where $\omega _3$ is the Euclidean volume of the unit ball in $\mathbb{R}^3$.
Note that 
$R(0)=0$, $R(s_0-\epsilon )< \infty$, and the derivative $R'(t)$ is well-defined and positive almost everywhere.

Define the function $\tilde{F}: B_{\mathrm{Eucl}}(0, R(s_0-\epsilon ) )\to \mathbb{R}$ by
$$
\tilde{F}:= t \text{ on } \partial B_{\mathrm{Eucl}}(0, R(t) ).
$$ 
Then for a.e. $t \in (0, s_0-\epsilon )$, 
$$
|\nabla \tilde{F}| = \frac{1}{R'(t)} \text{ on } \partial B_{\mathrm{Eucl}}(0, R(t) ).
$$ 
By (\ref{utile}), for some uniform constant $C>0$,
\begin{align}\label{symmetric}
\begin{split}
	C \int_{\Omega _{\epsilon , s_0}}|\nabla F|^2 \dvol_g & \geq \int_{0 }^{s_0-\epsilon } 3\omega_3\frac{R(t)^{2}}{R'(t)} dt\\
								& = \int_{0 }^{s_0-\epsilon  }\int_{ \partial B_{\mathrm{Eucl}}(0, R(t) )} |\nabla \tilde{F}|dA dt\\
								&= \int_{B_{\mathrm{Eucl}}(0, R(s_0-\epsilon ) )}|\nabla \tilde{F}|^2 dV.
\end{split}
\end{align}
The above inequality gives an upper bound of the capacity of the Euclidean ball $B_{\mathrm{Eucl}}(0, R(s_0-\epsilon ) )$.

Let us recall the definition and some properties of the capacity (see \cite[Definition 4.10, Theorem 4.15]{EG15}).
Set
$$
K:= \{f: \mathbb{R}^n\to \mathbb{R}: f \geq 0, f \in L^{2^*}(\mathbb{R}^n), |\nabla f| \in L^{2}(\mathbb{R}^n;\mathbb{R}^n)\} .
$$ 
For any open subset $A \subset \mathbb{R}^n$, the capacity of $A$ is defined as
$$
\mathrm{Cap}(A):= \inf \{ \int_{\mathbb{R}^n}|\nabla f|^{2}dV: f \in K, A \subset \{f \geq 1\} ^{0}\}
$$ 
where $\{ f \geq 1\}^0$ denotes the interior of the set $\{f \geq 1\}$.
Then, for any $x \in \mathbb{R}^n$,
\begin{equation}\label{scale-cap}
\mathrm{Cap}( B_{\mathrm{Eucl}}(x, r) ) = r^{n-2}\mathrm{Cap}(B_{\mathrm{Eucl}}(1)) >0.
\end{equation}

In our case, modifying $\tilde{F}$, we define a test function $\hat{\varphi}: \mathbb{R}^3 \to \mathbb{R}$ by
$$
\hat{\varphi}:= \frac{s_0-\epsilon -\tilde{F}}{s_0- \frac{5}{2}\epsilon } \text{ on } B_{\mathrm{Eucl}}(0, R(s_0-\epsilon ) ),
$$
$$ \hat{\varphi}:= 0 \text{ on } \mathbb{R}^3\setminus B_{\mathrm{Eucl}}(0, R(s_0-\epsilon ) ).
$$ 
With this definition, we have $\hat{\varphi} \in K$ and 
$$B_{\mathrm{Eucl}}(0, R(\epsilon ) ) \subset \{\hat{\varphi}>1\} = \{ \hat{\varphi} \geq 1\}^0,$$
{where $\{ \hat{\varphi} \geq 1\}^0$ denotes the interior of the set $\{ \hat{\varphi} \geq 1\}$.}
 This implies that
$$
\mathrm{Cap}(B_{\mathrm{Eucl}}(0, R(\epsilon ) ) ) \leq \int_{\mathbb{R}^3}|\nabla \hat{\varphi}|^2 dV = \frac{1}{(s_0- \frac{5}{2}\epsilon )^2}\int_{B_{\mathrm{Eucl}}(0, R(s_0-\epsilon ) )}|\nabla \tilde{F}|^2 dV.
$$ 
Together with (\ref{symmetric}), we deduce
$$
\mathrm{Cap}(B_{\mathrm{Eucl}}(0, R(\epsilon ) ) ) \leq \frac{C}{(s_0- \frac{5}{2}\epsilon )^2}\int_{\Omega _{\epsilon , s_0}}|\nabla F|^2 \dvol_g.
$$ 
From (\ref{scale-cap}), $s_0 \geq 3\epsilon $ and Lemma \ref{gradient F}, 
$$
R(\epsilon ) \leq \frac{C m(g)}{\epsilon ^2}.
$$ 
Since $\Vol(\Omega _{s_0-\epsilon , s_0} \cap \{|\nabla F| \neq 0\} )= W(\epsilon )= \omega _3 R(\epsilon )^3$, we conclude:
$$
\Vol(\Omega _{s_0-\epsilon  ,s_0 } \cap \{|\nabla F| \neq 0\} ) \leq C\left(\frac{ m(g)}{\epsilon ^2} \right) ^{3}.
$$ 
This is the desired conclusion, up to renaming $s_0-\epsilon$ and $s_0$.
\vspace{1em}

\textbf{Case 2)}: $s_0 \in (0, 3\epsilon )$.

That case is completely similar to the first case. For any regular value  $s \in (s_0, 5\epsilon )$, we have
$$
\Vol(\Omega _{s_0, s})^{\frac{2}{3}} \leq 6C_{\mathrm{isop}}\Area(S_s).
$$
For $t \in (0, 5\epsilon -s_0)$, define
$$
W(t):= \Vol(\Omega _{s_0, s_0+t} \cap \{|\nabla F| \neq 0\} ),\ S(t):= \Area(S_{s_0+t}),
$$ 
and
$$
\omega _3 R(t)^{3}= W(t),
$$ 
with $R(0)=0$ and $R(5\epsilon - s_0) <\infty$.

Define $\tilde{F}: B_{\mathrm{Eucl}}(0, R(5\epsilon -s_0) ) \to \mathbb{R}$ by
$$
\tilde{F}:= t \text{ on } \partial B_{\mathrm{Eucl}}(0, R(t) ).
$$ 
Then as in Case 1), we get
$$
\int_{B_{\mathrm{Eucl}}(0, R(5\epsilon -s_0) )}|\nabla \tilde{F}|^2 dV \leq C \int_{\Omega _{s_0, 5\epsilon }}|\nabla F|^2 \dvol_g.
$$ 
Modifying $\tilde{F}$, we define a test function $\hat{\varphi}: \mathbb{R}^3\to \mathbb{R}$ by
$$
\hat{\varphi}:= \frac{5\epsilon -s_0- \tilde{F}}{\frac{7}{2}\epsilon -s_0} \text{ on } B_{\mathrm{Eucl}}(0, R(5\epsilon - s_0) ),
$$ 
$$\hat{\varphi}:= 0 \text{ on } \mathbb{R}^3\setminus B_{\mathrm{Eucl}}(0, R(5\epsilon -s_0) ).
$$ 
So $\hat{\varphi} \in K$ and $B_{\mathrm{Eucl}}(0, R(\epsilon ) ) \subset \{\hat{\varphi}>1\} = \{ \hat{\varphi} \geq 1\} ^0$. This implies that
$$
\mathrm{Cap}( B_{\mathrm{Eucl}}(0, R(\epsilon ) ) ) \leq \frac{1}{(\frac{7}{2}\epsilon - s_0)^2}\int_{B_{\mathrm{Eucl}}(0, R(5\epsilon -s_0) )}|\nabla \tilde{F}|^2 dV.
$$ 
Thus we have
$$
R(\epsilon ) \leq \frac{C m(g)}{\epsilon ^2},
$$ 
and since $\Vol(\Omega _{s_0, s_0+\epsilon } \cap \{|\nabla F| \neq 0\} ) = W(\epsilon ) = \omega _3 R(\epsilon )^3$, we conclude:
$$
\Vol(\Omega _{s_0, s_0+\epsilon  } \cap \{|\nabla F| \neq 0\} ) \leq C\cdot \left(\frac{ m(g)}{\epsilon ^2}\right) ^3.
$$ 
\end{proof}

\begin{lem}\label{small boundary}
	There exists a regular value $\tau_0 \in (0, 6\epsilon )$ such that $S_{\tau_0 }$ is a smooth surface satisfying
	$$
	\Area(S_{\tau_0 }) \leq C \left(\frac{m(g)}{\epsilon ^2}\right) ^2.
	$$ 
\end{lem}
\begin{proof}
If $\inf_{s \in (0, 5\epsilon )}\Area(S_s)=0, $ then we can take a regular value $\tau_0  \in (0, 5\epsilon )$ such that $S_{\tau_0 }$ is a smooth surface and satisfies 
	$$
	\Area(S_{\tau_0}) \leq \frac{m(g)^2}{\epsilon ^{4}}.
	$$

	If $\inf_{s \in (0, 5\epsilon )}\Area(S_s) >0$, then we can choose $s_0 \in (0, 5\epsilon )$ as in Lemma \ref{PFS-ineq}.  By the coarea formula,
	\begin{align}
		\begin{split}
			\int_{s_0 }^{s_0+\epsilon } \Area(S_t) dt &= \int_{\Omega _{s_0 , s_0+\epsilon }}|\nabla F|\\
										   & \leq \left( \int_{\Omega _{s_0, s_0+\epsilon }}|\nabla F|^2 \right) ^{\frac{1}{2}}\cdot (\Vol(\Omega _{s_0, s_0+\epsilon } \cap \{|\nabla F| \neq 0\} ) )^{\frac{1}{2}}\\
										   &\leq C \sqrt{m(g)} \left( \frac{m(g)}{\epsilon ^2} \right) ^{\frac{3}{2}},
		\end{split}
	\end{align}
	where in the last inequality we used Lemma \ref{gradient F} and Lemma \ref{PFS-ineq}. 
	So there exists a regular value $\tau_0 \in (s_0 , s_0+ \epsilon ) \subset (0, 6\epsilon )$ such that $S_{\tau_0 } $ is smooth and 
	$$
	\Area(S_{\tau_0 }) \leq \frac{C m(g)^2}{\epsilon ^{4}}.
	$$ 
	
\end{proof}

For $\tau _0$ as in the above lemma, the domain $\Omega _{0, \tau_0 }$ has smooth boundary $S_{\tau_0 }$, whose area is very small depending on $m(g)$, and contains the end of $M_{ext}$. In general, $\mathbf{u}: \Omega _{0, \tau _0}\to \mathbb{R}^3$ is only a local but not global diffeomorphism. For this reason, we need to restrict $\mathbf{u}$ to a smaller region.

By definition of asymptotic flatness and the construction of $\mathbf{u}$, we clearly have the following lemma (see for instance  \cite[Lemma 3.3, 3.4]{Dong22}).

\begin{lem}\label{one to one}
	$\mathbf{u}$ is one-to-one around the end at infinity. That is, for some big number $L>0$ (not uniform in general), there exists an unbounded domain $U \subset M_{ext} $ containing the end such that $\mathbf{u}: U \to \mathbb{R}^3\setminus B_{\mathrm{Eucl}}(0,L)$ is injective and onto.
\end{lem}

{Recall that $\epsilon$ was chosen at the beginning of this section, and that $|\Jac \mathbf{u}(x) - \Id| \leq 100\sqrt{\epsilon}$ on $\Omega_{0,6\epsilon}$.}

\begin{prop}\label{regular subregion}
	Assume $(M, g)$ is a complete asymptotically flat $3$-manifold with nonnegative scalar curvature. For a given end of $(M,g)$, let $M_{ext}$ be the exterior region associated to this end and let $m(g)$ be its mass. Then there exists a connected region $E\subset M_{ext}$ with smooth boundary, such that 
{
\begin{enumerate}
\item 	the restricted harmonic map  $$
	\mathbf{u}: E\to Y \subset \mathbb{R}^3
	$$
	is a diffeomorphism onto its image $Y:= \mathbf{u}(E)$, 
	\item $E$ contains the end of $M_{ext}$ and $Y$ contains the end of $\mathbb{R}^3$, 
	\item on $E$, the Jacobian matrix of $\mathbf{u}$ satisfies $|\Jac \mathbf{u}(x) - \Id| \leq 100\sqrt{\epsilon}$,
	\item 
	$
	\Area(\partial E) \leq \frac{C m(g)^{2 }}{\epsilon ^{4}}.
	$
		\end{enumerate}
	}
\end{prop}

\begin{proof}
Take $\tau_0 \in (0,6\epsilon)$ as in Lemma \ref{small boundary}. Then $\mathbf{u}: \Omega _{0,\tau_0 }\to \mathbb{R}^3$ is a local diffeomorphism, { $|\Jac \mathbf{u}(x) - \Id| \leq 100\sqrt{\epsilon}$}  and $\partial \Omega _{0,\tau_0 }=S_{\tau_0} $ has small area. Let $Y_1 \subset  \mathbf{u}(\Omega _{0,\tau_0 })$ be the connected component containing the end of $\mathbb{R}^3$, and 
	$$
	Y_2:= \{y \in Y_1: \mathcal{H}^0(\mathbf{u}^{-1}\{y\} \cap \Omega _{0,\tau_0 }) =1\} .
	$$ 
	By Lemma \ref{one to one}, $Y_2 \neq \emptyset$ and contains the end of $\mathbb{R}^3$. Notice that $Y_2$ is open in $Y_1$ and $\partial Y_2 \subset \mathbf{u}(S_{\tau_0} )$.

	Since $\mathbf{u}(S_{\tau_0} )$ is a smooth immersed surface in $\mathbb{R}^3$, we can choose a slightly smaller region $Y_3$ such that $Y_3\cup \partial Y_3 \subset Y_2$, $\partial Y_3$ is a smooth embedded surface and $\Area(\partial Y_3) \leq 2 \Area(\partial Y_2)$. 

	Define $Y$ to be the connected component of $Y_3$ containing the end of $\mathbb{R}^3$, and $E:= \mathbf{u}^{-1}(Y)$. By construction, $\mathbf{u}: E \to Y$ is a diffeomorphism. Moreover by (\ref{abu}),
	$$
	\Area(\partial E) = \int_{\partial Y}|\Jac (\mathbf{u}\big|_{\partial E})^{-1}| \leq C(1+\epsilon' )\Area(\partial Y) \leq C\Area(S_{\tau_0} ).
	$$ 
	
\end{proof}

\section{Pointed measured Gromov-Hausdorff convergence} \label{Section 4}

In general a regular region $E$ such as the one given by Proposition \ref{regular subregion} is not sufficient to get Gromov-Hausdorff convergence.  
In this section, we will construct a more refined subregion over which we have  pointed measured Gromov-Hausdorff convergence for the induced length metric. 
{
Qualitatively, here is the general result which we will end up proving:
\begin{thm}
Let $\hat{\varepsilon}_i$ be a sequence of positive numbers converging to $0$. Consider a sequence of complete Riemannian $3$-manifolds $(E_i,g_i)$ with compact boundaries and assume that for each $i$,  there is an embedding
$$\mathbf{u}_i:E_i\to \mathbb{R}^3$$ 
such that
\begin{enumerate}
	\item $\mathbf{u}_i(E_i)$ contains the end of $\mathbb{R}^3$,
	\item on $E_i$, the Jacobian matrix satisfies $|\Jac \mathbf{u}_i(x) - \Id| \leq \hat{\varepsilon}_i$,
	\item 
	$
	\Area(\partial E_i) \leq \hat{\varepsilon}_i.
	$
		\end{enumerate}
	Then there is a closed subset $E''_i\subset E_i$ with compact boundary, such that for any choice of basepoint $x_i\in E''_i$, the pointed sequence $(E''_i,\hat{d}_{g_i,E''_i},x_i)$ converges in the pointed measured Gromov-Hausdorff topology to the flat Euclidean $3$-space. Here, $E''_i$ is endowed with the length metric $\hat{d}_{g_i,E''_i}$ induced by $g_i$, where the distance between two points is measured using paths inside $E''_i$. 
\end{thm}
}
\begin{rmk}
	{
	The higher-dimensional version of this result also holds thanks to an induction argument and a slight modification of the proof provided below. For more details, see \cite[Appendix]{Dong24}.  }
\end{rmk}

{
For concreteness, we will show the result above in our special setting. 
We use the notations $C$, $\Psi(.)$, $\Psi(.|D)$ as explained in Section \ref{notations}. 
}

Fix a continuous function $\xi :(0, \infty)\to (0,\infty)$ with $\xi (0)=0$ and
$$\lim_{x\to 0^+}{\xi (x)} =0.$$
All our constructions in this section will depend on this function $\xi $. For our purposes we can assume without loss of generality that
$$\lim_{x\to 0^+}\frac{x}{\xi (x)} =0.$$
Choose two other continuous functions $\xi_0 , \xi _1: (0, \infty)\to (0, \infty) $ with 
{
$$\lim_{x\to 0^+}{\xi_0 (x)} = \lim_{x\to 0^+}{\xi_1 (x)}=0,$$
}
and
$$
\lim_{x\to 0^+}\frac{\xi (x)}{\xi _0^{100}(x)}= \lim_{x\to 0^+}\frac{\xi _0(x)}{\xi _1^{100}(x)}=0.
$$ 
The reader can think of these functions $\xi,\xi_0,\xi_1$, as converging to $0$  very slowly as $x\to 0^+$. 
Set 
$$
\delta _0:= \xi_0 (m(g) ),\quad \delta _1:= \xi _1(m(g) ).$$ Then 
\begin{equation}\label{aum}
\delta _1^{100} \gg \delta _0\gg \xi(m(g))^{\frac{1}{100}}\gg m(g)^{\frac{1}{100}}
\end{equation}
when $0<m(g) \ll 1$. In the following, we will always assume $0<m(g)\ll 1$.

Let $E$ be the regular region given by Proposition \ref{regular subregion} and its image $Y:=\mathbf{u}(E) \subset \mathbb{R}^3$. Then $\mathbf{u}: E\to Y$ is a diffeomorphism and
\begin{equation} \label{partialy}
\Area(\partial Y) \leq \frac{C m(g)^{2}}{\epsilon ^{4}}.
\end{equation}

For any subset  $U\subset E$, let $(U, \hat{d}_{g,U})$ be the induced length metric on $U$ of the metric $g$, that is, for any $x_1,x_2 \in U$ ,
$$
\hat{d}_{g,U}(x_1,x_2):= \inf \{L_g(\gamma ):\gamma \text{ is a rectifiable curve connecting }x_1,x_2 \text{ and }\gamma \subset U \} ,
$$
where $L_g(\gamma ) = \int_{0 }^1|\gamma '|_g$ is the length of $\gamma $ with respect to metric $g$.  By convention, two points in two different path connected components of $U$ are at infinite $\hat{d}_{g,U}$-distance. 

Similarly, for any $V \subset Y \subset \mathbb{R}^3$, let $(V, \hat{d}_{\mathrm{Eucl}, V})$ be the induced length metric on $V$ of the standard Euclidean metric $g_\mathrm{Eucl}$, that is, for any $y_1, y_2 \in V$,
$$
\hat{d}_{\mathrm{Eucl}, V}(y_1,y_2):= \inf \{L_{\mathrm{Eucl}}(\gamma ):  \gamma \text{ is a rectifiable curve connecting }y_1,y_2 \text{ and }\gamma \subset V\} ,
$$ 
where $L_{\mathrm{Eucl}}(\gamma )$ is the length of $\gamma $ with respect to the Euclidean metric $g_{\mathrm{Eucl}}$. As before, $\hat{d}_{\mathrm{Eucl}, V}$ can be infinite for pairs of points in different path connected components.

Write 
$$\Sigma := \partial Y \subset \mathbb{R}^3$$ and 
let $\mathcal{W}$ be the compact domain bounded by $\Sigma$ in $\mathbb{R}^3$, so $\partial \mathcal{W} = \Sigma.$ 

The main part of this section is devoted to the proofs of Proposition \ref{Y'-GH} and the lemmas leading to it. In Proposition \ref{Y'-GH}, we construct a subregion inside $Y\subset \mathbb{R}^3$ with small area boundary, over which the induced length metric of $g_{\mathrm{Eucl}}$ is close to the restriction of the Euclidean metric.

By the isoperimetric inequality and (\ref{partialy}), 
\begin{equation}\label{isooop1}
\Vol(\mathcal{W}) \leq C \Area(\Sigma)^{\frac{3}{2}}\leq  \frac{C m(g)^{3 }}{\epsilon ^{6 }}.
\end{equation}

Take $$\epsilon = \delta _0$$ (in particular in this section $\epsilon$ depends on $m(g)$). Recall that 
$$\epsilon':=100\sqrt{\epsilon} \ll 1.$$
Then
\begin{equation}\label{qqqqq}
\Area(\Sigma) \leq \frac{C m(g)^2}{\delta _0^4 }\ll 1,\ \Vol(\mathcal{W}) \leq \frac{C m(g)^3}{\delta _0^6}\ll 1.
\end{equation}

For any triple $\mathbf{k}=(k_1,k_2,k_3)\in \mathbb{Z}^3$, consider the {closed} cube $\mathbf{C}_{\mathbf{k}}(\delta _1)$ defined by
$$
{\mathbf{C}_{\mathbf{k}}(\delta _1):= [k_1\delta _1, (k_1+1)\delta _1 ] \times [k_2\delta _1, (k_2+1)\delta _1 ] \times [k_3\delta _1, (k_3+1)\delta _1 ] \subset \mathbb{R}^3.}
$$ 
For $t\in\mathbb{R}$, define the plane $$A_{\mathbf{k}, \delta _1}(t):= \{(x_1, x_2, x_3) \in \mathbb{R}^3: x_3= (k_3 +t) \delta _1\}.$$ 
By definition $\mathbf{C}_{\mathbf{k}}(\delta _1) \subset  \bigcup_{t\in[0,1]} A_{\mathbf{k}, \delta _1}(t)$. 

By (\ref{qqqqq}) and the coarea formula, there exists 
$$t_{\mathbf{k}} \in (\frac{1}{2}, \frac{1}{2}+ \delta _0)$$ such that ${A_{\mathbf{k}, \delta _1}(t_{\mathbf{k}}) \cap \Sigma }$ consists of smooth curves and 
	{ 
\begin{align}\label{small length k}
\begin{split}
 \Length( A_{\mathbf{k},\delta _1}(t_{\mathbf{k}}) \cap \Sigma \cap \mathbf{C}_{\mathbf{k}}(\delta _1) ) &\leq \frac{\Area(\Sigma \cap \mathbf{C}_{\mathbf{k}}(\delta _1) )}{\delta _0 \delta _1}\\
										    &\leq \frac{C m(g)^2}{ \delta _0^5 \delta _1}\\
										    & \leq m(g).
\end{split}
\end{align}}

{ The square $A_{\mathbf{k},\delta _1}(t_{\mathbf{k}}) \cap \mathbf{C}_{\mathbf{k}}(\delta _1) $ has side length $\delta_1$, and the intersection $A_{\mathbf{k},\delta _1}(t_{\mathbf{k}}) \cap \Sigma \cap \mathbf{C}_{\mathbf{k}}(\delta _1) $ is a union of curves in this square, with total length much smaller than $\delta_1$ because of (\ref{small length k}) and (\ref{aum}), see Figure \ref{fig}. 
Set 
$$D'_{\mathbf{k}}$$
to be the path connected component of the complement  $(A_{\mathbf{k},\delta _1}(t_{\mathbf{k}}) \cap \mathbf{C}_{\mathbf{k}}(\delta _1)) \setminus \Sigma$ which has largest area compared to the other components (see Figure \ref{fig}), and set
$$D''_{\mathbf{k}} := (A_{\mathbf{k},\delta _1}(t_{\mathbf{k}}) \cap \mathbf{C}_{\mathbf{k}}(\delta _1)) \setminus (\Sigma \cup D'_{\mathbf{k}}).$$

 
}



{Note that each connected curve in $A_{\mathbf{k},\delta _1}(t_{\mathbf{k}}) \cap \Sigma \cap \mathbf{C}_{\mathbf{k}}(\delta _1) $ is contained in the boundary of at most two connected components of $D_{\mathbf{k}}''$. For each component $\mathscr{C}$ of $D_{\mathbf{k}}''$, the boundary $\partial \mathscr{C}$ is decomposed into a part inside the interior of the square $A_{\mathbf{k},\delta _1}(t_{\mathbf{k}}) \cap \mathbf{C}_{\mathbf{k}}(\delta _1)$, which has length much smaller than $\delta_1$ by the previous paragraph, and a part inside the boundary of the square. So since $\mathscr{C}$ is by definition not a component of $(A_{\mathbf{k},\delta _1}(t_{\mathbf{k}}) \cap \mathbf{C}_{\mathbf{k}}(\delta _1)) \setminus \Sigma$ with largest area, $\mathscr{C}$ touches at most two sides of the square and the part of $\partial \mathscr{C}$ in the boundary of the square is small too. In fact 
	$$\mathrm{Length}(\partial \mathscr{C} \cap \partial \mathbf{C}_{\mathbf{k}}(\delta _1)  )\leq 2\cdot  \mathrm{Length} ( A_{\mathbf{k},\delta _1}(t_{\mathbf{k}}) \cap \Sigma \cap \mathbf{C}_{\mathbf{k}}(\delta _1)). $$ 
By the usual isoperimetric inequality applied to each component $\mathscr{C}$ of $D_{\mathbf{k}}''$ and summing up those inequalities, we obtain from (\ref{small length k}):}
{ 
\begin{align}\label{aread''}
\begin{split}
	\Area(D_{\mathbf{k}}'') &\leq C \Length(A_{\mathbf{k},\delta _1}(t_{\mathbf{k}}) \cap \Sigma \cap \mathbf{C}_{\mathbf{k}}(\delta _1) ) ^2\\
				&\leq C m(g) \frac{ \Area(\Sigma \cap \mathbf{C}_{\mathbf{k}}(\delta _1) )}{\delta _0 \delta _1}\\
				&\leq \Area(\Sigma \cap \mathbf{C}_{\mathbf{k}}(\delta _1) ).
				\end{split}
\end{align}
}
 Let $\pi_{\mathbf{k}} : \mathbb{R}^3\to A_{\mathbf{k}, \delta _1}(t_{\mathbf{k}})$ be the orthogonal projection. Define 
 	{ 
 $$\mathbf{C}_{\mathbf{k}}(\delta _1)':= D_{\mathbf{k}}' \cup \left( \mathbf{C}_{\mathbf{k}}(\delta _1) \cap \pi_{\mathbf{k}}^{-1} (D_{\mathbf{k}}' \setminus \pi_\mathbf{k}(\Sigma \cap \mathbf{C}_{\mathbf{k}}(\delta _1)) ) \right) .$$
} 
For instance, if
$\Sigma \cap \mathbf{C}_{\mathbf{k}}(\delta _{1}) =\varnothing $, 
then
$\mathbf{C}_{\mathbf{k}}(\delta _{1})' = \mathbf{C}_{\mathbf{k}}(
\delta _{1})$, see Figure \ref{fig}.

\begin{figure}[H]
	\centering
	\includegraphics[width=0.8\textwidth]{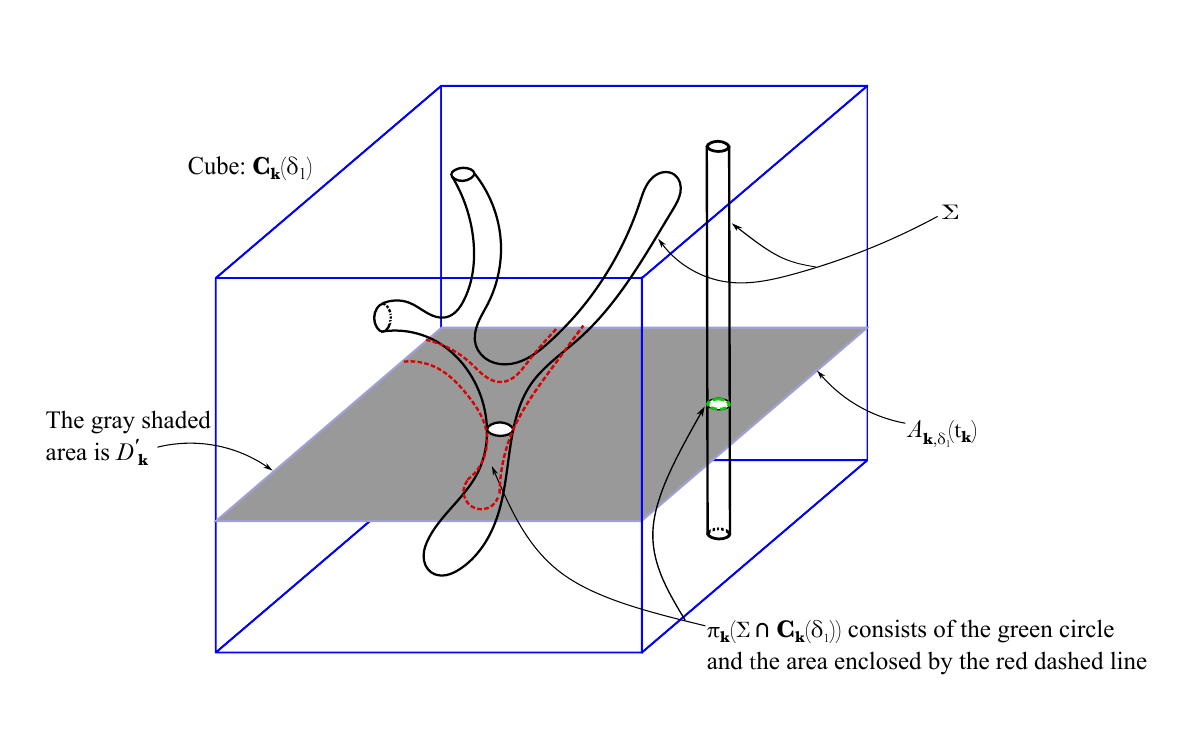}
	\caption{Construction of $\mathbf{C}_{\mathbf{k}}(\delta _1)'$. Note that in general, $D_{\mathbf{k}}''$ is not contained in $\pi_\mathbf{k}(\Sigma \cap \mathbf{C}_{\mathbf{k}}(\delta _1))$ or vice versa.}
	\label{fig}
\end{figure}

\begin{lem}\label{ooo}
	
{	$\Vol( \mathbf{C}_{\mathbf{k}}(\delta _1)\setminus \mathbf{C}_{\mathbf{k}}(\delta _1)') \leq 2 \delta _1 \Area(\Sigma \cap \mathbf{C}_{\mathbf{k}}(\delta _1) ) \leq m(g) \delta_1^3.$	
}
\end{lem}

\begin{proof}
	Since 
		{ 
	$$
	\Area(\pi_{\mathbf{k}}(\Sigma \cap \mathbf{C}_{\mathbf{k}}(\delta _1) ) ) \leq \Area(\Sigma \cap \mathbf{C}_{\mathbf{k}}(\delta _1)),
	$$ 
	and since by (\ref{aread''}),
	$$
	\Area(D_{\mathbf{k}}'') \leq   \Area(\Sigma \cap \mathbf{C}_{\mathbf{k}}(\delta _1) ), 
	$$ 
	we have
\begin{align}
\begin{split}
	\Vol(\mathbf{C}_{\mathbf{k}}(\delta _1) ) \setminus
	 \Vol(\mathbf{C}_{\mathbf{k}}(\delta _1)') &\leq 2\delta _1 \Area(\Sigma \cap \mathbf{C}_{\mathbf{k}}(\delta _1)).
\end{split}
\end{align}
}
We can use (\ref{partialy}) to conclude the proof.
\end{proof}

{
\begin{lem}
	$\mathbf{C}_{\mathbf{k}}(\delta _1)'$ is path connected, and 
	$$\mathbf{C}_{\mathbf{k}}(\delta _1)' \subset Y.$$
\end{lem}
\begin{proof}
	By definition, for any point $x \in \mathbf{C}_{\mathbf{k}}(\delta _1) \cap \pi_{\mathbf{k}}^{-1} (D_{\mathbf{k}}' \setminus \pi_\mathbf{k}(\Sigma \cap \mathbf{C}_{\mathbf{k}}(\delta _1)) )$, the line segment $L_x \subset \mathbf{C}_{\mathbf{k}}(\delta _1)$ through $x$ and orthogonal to $A_{\mathbf{k}, \delta _1}(t_{\mathbf{k}})$ satisfies $L_x \cap D_{\mathbf{k}}' \neq \emptyset$. Since $D_{\mathbf{k}}'$ is path connected, $\mathbf{C}_{\mathbf{k}}(\delta _1)'$ is also path connected.
	Next, assume towards a contradiction that $\mathbf{C}_{\mathbf{k}}(\delta _1)' \nsubseteq Y$. Since $\mathbf{C}_{\mathbf{k}}(\delta _1)'$ is path connected, and by its definition, $\Sigma$ has no intersection with $\mathbf{C}_{\mathbf{k}}(\delta _1)'$, we should have $\mathbf{C}_{\mathbf{k}}(\delta _1)' \subset \mathcal{W}$.
	From the volume estimate Lemma  \ref{ooo}, 
	$$\Vol(\mathbf{C}_{\mathbf{k}}(\delta _1)') \geq  (1-m(g) ) \delta _1^3.$$
	But recall from (\ref{isooop1}) that
	$$
	\Vol(\mathcal{W}) \leq \frac{C m(g)^3}{\delta _0^6},
	$$ 
	which gives 
	$$
	(1-m(g) ) \delta _1^{3} \leq \frac{C m(g)^3}{\delta _0^6},
	$$
	a contradiction since $0<m(g) \ll 1$.
\end{proof}
}

Define 
	$$
	Y':= \cup _{\mathbf{k}\in \mathbb{Z}^3}\mathbf{C}_{\mathbf{k}}(\delta _1)' \subset Y.
	$$ 
	Notice that when $|\mathbf{k}|$ is big enough, one can certainly ensure that $\mathbf{C}_{\mathbf{k}}(\delta _1)'= \mathbf{C}_{\mathbf{k}}(\delta _1)$, so that $Y\setminus Y'$ is a bounded set. It could be shown that $Y'$ is path connected, but we will not need it.


	For any subset $V \subset Y$, let $V_t$ be the $t$-neighborhood of $V$ inside $(Y, \hat{d}_{\mathrm{Eucl},Y})$ in terms of the length metric $\hat{d}_{\mathrm{Eucl},Y}$, i.e.
	$$
	V_{t }:= \{y \in Y: \exists z \in V \text{ such that }\hat{d}_{\mathrm{Eucl},Y}(y, z) \leq t  \} .
	$$ 
	So $(Y')_t$ is the $t$-neighborhood of $Y'$ inside $(Y, \hat{d}_{\mathrm{Eucl}, Y})$.

	In the following lemma, by modifying some $(Y')_t$, we construct a domain with smooth boundary such that its boundary area is small and it is very close to $Y'$ in the Gromov-Hausdorff topology with respect to a length metric.

\begin{lem} \label{ybis}
	There exists a closed subset $Y''$ with smooth boundary such that $Y'\subset Y'' \subset (Y')_{6 \delta _0} $,
	$$
	\Area(\partial Y'' ) \leq \frac{ m(g)^2}{\delta _0^5},
	$$ 
	and
	$Y''$ is contained in the $6\delta_0$-neighborhood of $Y'$ inside $Y''$, with respect to its length metric $\hat{d}_{\mathrm{Eucl}, Y''}$.
	

\end{lem}
\begin{proof}
	Smoothing the Lipschitz function $\hat{d}_{\mathrm{Eucl}, Y}(Y', \cdot)$, we can get a smooth function $\phi : Y \to \mathbb{R}$ such that $ |\phi - \hat{d}_{\mathrm{Eucl}, Y}(Y', \cdot ) | \leq \delta _0$ and $|\nabla \phi | \leq 2$ (see for instance \cite[Proposition 2.1]{GW79}). Applying coarea formula to $\phi $, we have
	$$
	\int_{3\delta _0}^{4 \delta _0}\Area( \phi ^{-1}(t) \cap Y ) dt = \int_{ \{3\delta _0 < \phi < 4\delta _0\} \cap Y} |\nabla \phi | \dvol \leq 2 \Vol(Y\setminus Y').
	$$ 
By Lemma \ref{ooo}, for each $\mathbf{k}\in \mathbb{Z}^3$,
$$
0 \leq \Vol(\mathbf{C}_{\mathbf{k}}(\delta _1)) - \Vol(\mathbf{C}_{\mathbf{k}}'(\delta _1)) \leq 2 \delta _1 \Area(\Sigma \cap \mathbf{C}_{\mathbf{k}}(\delta _1)).
$$
Since the number of overlaps of $\{\mathbf{C}_{\mathbf{k}}(\delta _1)\} _{\mathbf{k}\in \mathbb{Z}^3}$ is uniformly bounded,   
\begin{align}
\begin{split}
	0 \leq \Vol(Y)- \Vol(Y')& \leq 2 \delta _1 \sum_{\mathbf{k}\in \mathbb{Z}^3}  \Area(\Sigma \cap \mathbf{C}_{\mathbf{k}}(\delta _1))\\
				& \leq C \delta _1\Area(\Sigma)\\
				& \leq \frac{C\delta _1 m(g)^2}{\delta _0^4}.
\end{split}
\end{align}
	So we can find a regular value $t \in (3\delta _0, 4 \delta _0)$ of $\phi $ such that $\phi ^{-1}(t)$ is smooth and 
	$$\Area(\phi ^{-1}(t) \cap Y) \leq \frac{C \delta _1 m(g)^2}{\delta _0^5} \leq \frac{m(g)^2}{8\delta _0^5}.$$
	Smoothing $(\phi ^{-1}(t) \cap Y ) \cup (\partial Y \cap \{\phi < t\} )$ inside $Y$, we can get a smooth closed surface $S_1$ such that 
	{
	$$ S_1 \subset  (Y')_{5 \delta _0},$$
	}
	$$
	\Area(S_1) \leq 2(\Area( \phi ^{-1}(t) \cap Y) + \Area(\partial Y) ) \leq \frac{m(g)^2}{4 \delta _0^5} + \frac{C m(g)^2}{\delta _0^4} \leq \frac{m(g)^2}{ 2 \delta _0^5}
	$$
	and there is  a closed region $Y_1$ 
	satisfying
	\begin{equation}\label{subsubsub}
	Y' \subset Y_1 \subset (Y')_{5 \delta _0} \subset Y\quad \text{and}\quad \partial Y_1 \subset S_1.
	\end{equation}

	At this point, $Y_1$ is close to $Y'$ in the Hausdorff topology with respect to $\hat{d}_{\mathrm{Eucl},Y}$, but possibly not with respect to its own length metric $\hat{d}_{\mathrm{Eucl},Y_1}$.
	To remedy this, choose a finite subset $\{x_j\} $ consisting of $\delta _0$-dense discrete points of $(Y_1\setminus Y', \hat{d}_{\mathrm{Eucl}, Y_1})$ and denote by $\gamma _j \subset Y$ a smooth curve connecting $x_j$ to $Y'$ with minimal length with respect to the length metric $\hat{d}_{\mathrm{Eucl},Y}$. Then by (\ref{subsubsub}), $\gamma _j$ has length at most $5\delta_0$, and so $\gamma _j \subset (Y')_{5 \delta _0}$. By thickening each $\gamma _j$, we can get thin solid tubes $T_j$ inside $\delta _0$-neighborhood of $\gamma _j$ with arbitrarily small boundary area.
	Let $Y_2:= Y_1 \cup (\cup _{j} T_j)$. By smoothing the corners of $Y_2$ and taking the closure, we have a closed domain $Y''$ with smooth boundary such that 
	$$Y' \subset Y'' \subset Y_2\subset Y'_{6 \delta _0}$$ 
	and
	$$
	\Area(\partial Y'') \leq 2\Area(S_1) \leq \frac{m(g)^2}{\delta _0^5}.
	$$

	For any $y \in Y''\setminus Y'$, by our construction, there exists some $j$ such that either $ \hat{d}_{\mathrm{Eucl},Y_1}(y, x_j) \leq \delta _0$ or $y \in T_j$. In each case, there exists a smooth curve $\sigma _{y,j}\subset Y''$ connecting $y$ to a point in $\gamma _j$ and $\Length(\sigma _{y, j}) \leq \delta _0$. Since $\Length (\gamma _j) \leq 5 \delta _0$, $\sigma _{y, j} \cup \gamma _j$ is a piecewise smooth curve inside $Y''$ connecting $y$ to $Y'$ with length smaller than $6\delta _0$.
	So inside the length space $(Y'', \hat{d}_{\mathrm{Eucl},Y''})$, $Y''$ is in the $6\delta _0$-neighborhood of $Y'$ as desired.

%

\end{proof}

Let $Y''$ be as in Lemma \ref{ybis}. We have not yet shown that $Y''$ is path connected; this will be a consequence of Proposition \ref{Y'-GH}.
Recall that $\hat{d}_{\mathrm{Eucl},Y''}$ is defined as the length metric on $Y''$ induced by $g_{\mathrm{Eucl}}$. Since $Y' \subset Y'' \subset Y$, we have $d_{\mathrm{Eucl}} \leq \hat{d}_{\mathrm{Eucl},Y} \leq \hat{d}_{\mathrm{Eucl},Y''}$. 

\begin{lem}\label{local small diam}
	$\diam_{\hat{d}_{\mathrm{Eucl}, Y''}}(\mathbf{C}_{\mathbf{k}}(\delta _1)' ) \leq 5 \delta _1.$
\end{lem}
\begin{proof}
	For any two points $x_1, x_2 \in \mathbf{C}_{\mathbf{k}}(\delta _1)'$, let $L_{x_1}, L_{x_2}$ be the line segments inside $\mathbf{C}_{\mathbf{k}}(\delta _1)$ through $x_1, x_2$ and orthogonal to $A_{\mathbf{k},\delta _1}(t_{\mathbf{k}})$ respectively. Let $x_1' = L_{x_1}\cap D_{\mathbf{k}}', x_2' = L_{x_2}\cap  D_{\mathbf{k}}'$. 
	{Consider the line segment between $x_1', x_2'$ and perturb it with fixed endpoints, to get a path $P $ joining $x_1'$ to $x_2'$ with length at most $d_{\mathrm{Eucl}}(x_1',x_2') +  m(g)$ and intersecting $A_{\mathbf{k}, \delta _1}(t_{\mathbf{k}}) \cap \Sigma \cap \mathbf{C}_{\mathbf{k}}(\delta _1)$ transversally. If this path is not inside $D_{\mathbf{k}}'$, we can replace each component of $P \cap A_{\mathbf{k}, \delta _1}(t_{\mathbf{k}}) \cap \mathcal{W} \cap \mathbf{C}_{\mathbf{k}}(\delta _1)$
by an embedded curve with same endpoints but contained in  $A_{\mathbf{k}, \delta _1}(t_{\mathbf{k}}) \cap \Sigma \cap \mathbf{C}_{\mathbf{k}}(\delta _1)$, and we get in this way a new path $\tilde{P}$. By (\ref{small length k}), the length of  $\tilde{P}$ can be chosen to be smaller than $d_{\mathrm{Eucl}}(x_1',x_2') +  2m(g)$. We perturb $\tilde{P}$ one more time into a path $\gamma$, so that $\gamma$ sits inside  $D_{\mathbf{k}}'$: now $\gamma$ is a curve between $x_1', x_2'$ inside $D_{\mathbf{k}}'$ such that 
\begin{align*}
	\Length_{\mathrm{Eucl}}(\gamma ) 
					 \leq d_{\mathrm{Eucl}}(x_1',x_2') + 3 m(g).			 
\end{align*}
}
	Consider the curve $\tilde{\gamma }$ consisting of three parts: the line segment $[x_1x_1']$ between $x_1, x_1'$, the curve $\gamma $, and the line segment $[x_2'x_2]$ between $x_2', x_2$. We have $\tilde{\gamma } \subset \mathbf{C}_{\mathbf{k}}(\delta _1)' \subset Y'$, so
	$$
	\hat{d}_{\mathrm{Eucl}, Y''}(x_1, x_2) \leq L_{\mathrm{Eucl}}(\tilde{\gamma }) \leq 4 \delta _1 + 3m(g) \leq 5 \delta _1.
	$$ 
\end{proof}

\begin{lem} \label{ajout}
	For any basepoint $q \in Y'$ and any $D>0$,
	$$
	d_{pGH}( (Y' \cap B_{\mathrm{Eucl}}(q, D), \hat{d}_{\mathrm{Eucl},Y''},q), (Y' \cap B_{\mathrm{Eucl}}(q, D), d_{\mathrm{Eucl}},q) ) \leq \Psi (m(g)).
$$ 
\end{lem}
\begin{proof}
	Let $x_0, y_0 \in Y'\cap B_{\mathrm{Eucl}}(q, D)$ be two points and $x_0 \in \mathbf{C}_{\mathbf{k}}(\delta _1)', y_0 \in \mathbf{C}_{\mathbf{l}}(\delta _1)'$ for some $\mathbf{k}, \mathbf{l} \in \mathbb{Z}^3$. Since $d_{\mathrm{Eucl}} \leq \hat{d}_{\mathrm{Eucl}, Y'' }$, it's enough to show 
	\begin{equation}\label{jung}
	\hat{d}_{\mathrm{Eucl}, Y''}(x_0, y_0) \leq d_{\mathrm{Eucl}}(x_0, y_0) + \Psi (m(g)).
	\end{equation}

	Let $T_{\mathbf{k}, \mathbf{l}}$ be the translation which maps $\mathbf{C}_{\mathbf{k}}(\delta _1)$ to $\mathbf{C}_{\mathbf{l}}(\delta _1)$. Then by Lemma \ref{ooo},
{
\begin{equation}\label{cloak}
\begin{split}
	\Vol( T_{\mathbf{k},\mathbf{l}}(\mathbf{C}_{\mathbf{k}}(\delta _1)') \cap \mathbf{C}_{\mathbf{l}}(\delta _1)') 
	\geq & \Vol(\mathbf{C}_{\mathbf{k}}(\delta _1)) -(\Vol(\mathbf{C}_{\mathbf{k}}(\delta _1) \setminus \mathbf{C}_{\mathbf{k}}(\delta _1)'))
	\\ &-(  \Vol(\mathbf{C}_{\mathbf{l}}(\delta _1) \setminus \mathbf{C}_{\mathbf{l}}(\delta _1)'))\\
												     \geq & (1- 2 m(g) ) \delta _1^3. 
\end{split}
 \end{equation}
}

 If $\mathbf{k}= \mathbf{l}$, then by Lemma \ref{local small diam}, we know that
 $$\hat{d}_{\mathrm{Eucl},Y'' }(x_0,y_0) \leq 5\delta _1 \leq  d_{\mathrm{Eucl}}(x_0,y_0)+ 5 \delta _1.$$

{
 Suppose that  $\mathbf{k}\neq \mathbf{l}$. Recall that 
 $$\mathbf{C}_{\mathbf{k}}(\delta _1)' \subset Y'\subset Y''.$$
We claim that there is at least one point $x_0' \in \mathbf{C}_{\mathbf{k}}(\delta _1)' $ such that $T_{\mathbf{k}, \mathbf{l}}(x_0') \in \mathbf{C}_{\mathbf{l}}(\delta _1)'$ and the straight line segment $[x_0' T_{\mathbf{k}, \mathbf{l}}(x_0')]$ between these two points has no intersection with $\partial Y''$. Assume otherwise, then it means that if $x\in \mathbf{C}_{\mathbf{k}}(\delta _1)' $ and $T_{\mathbf{k}, \mathbf{l}}(x) \in \mathbf{C}_{\mathbf{l}}(\delta _1)'$,
the line containing $x$ and $T_{\mathbf{k}, \mathbf{l}}(x)$ intersects $\partial Y''$.
For any $x\in \mathbb{R}^3$, the straight line $\mathcal{L}_{x,T_{\mathbf{k}, \mathbf{l}}(x)}$ containing $x$ and $T_{\mathbf{k}, \mathbf{l}}(x)$ meets the set $T_{\mathbf{k},\mathbf{l}}(\mathbf{C}_{\mathbf{k}}(\delta _1)') \cap \mathbf{C}_{\mathbf{l}}(\delta _1)'$ in a subset of total length at most say $3\delta_1$:
\begin{equation} \label{vendredi}
\Length(\mathcal{L}_{x,T_{\mathbf{k}, \mathbf{l}}(x)} \cap T_{\mathbf{k},\mathbf{l}}(\mathbf{C}_{\mathbf{k}}(\delta _1)') \cap \mathbf{C}_{\mathbf{l}}(\delta _1)') \leq 3\delta_1
\end{equation}
(this is simply due to the fact that in the Euclidean space $\mathbb{R}^3$, the cube $\mathbf{C}_{\mathbf{l}}(\delta _1)$ which contains $\mathbf{C}_{\mathbf{l}}(\delta _1)'$ has Euclidean diameter at most $3\delta_1$).
By the coarea formula, we  would get
\begin{equation}\label{samedi}
 \Vol( T_{\mathbf{k},\mathbf{l}}(\mathbf{C}_{\mathbf{k}}(\delta _1)') \cap \mathbf{C}_{\mathbf{l}}(\delta _1)') \leq 3\delta_1 \Area(\partial Y'' ).
\end{equation}
To see this, consider the orthogonal projection $\Pi_{\mathbf{k}, \mathbf{l}}:\mathbb{R}^3\to \mathbb{R}^2$ onto any $2$-plane inside $\mathbb{R}^3$ which is orthogonal to the translation vector of $T_{\mathbf{k}, \mathbf{l}}$. 
{Set $$U: =  T_{\mathbf{k},\mathbf{l}}(\mathbf{C}_{\mathbf{k}}(\delta _1)') \cap \mathbf{C}_{\mathbf{l}}(\delta _1)'.$$ Then applying the coarea formula (c.f. \cite[Theorem 3.10]{EG15}) to $\Pi_{\mathbf{k}, \mathbf{l}}$,  we have 
\begin{align*}
\mathrm{Vol}(U) = \int_{ \{ y \in \mathbb{R}^2: U \cap \Pi_{\mathbf{k}, \mathbf{l}}^{-1}\{y\} \neq \emptyset\} } \mathrm{Length} \left( U \cap \Pi_{\mathbf{k}, \mathbf{l}}^{-1} \{y\}  \right) dy.	
\end{align*}
By (\ref{vendredi}), $\mathrm{Length}(U \cap \Pi_{\mathbf{k}, \mathbf{l}}^{-1}\{y\} ) \leq 3 \delta _1$. Notice that for all $ y \in \mathbb{R}^2$, if $U \cap \Pi_{\mathbf{k}, \mathbf{l}}^{-1}\{y\} \neq \emptyset$, then $\Pi_{\mathbf{k}, \mathbf{l}}^{-1}\{y\} = \mathcal{L}_{x, T_{\mathbf{k}, \mathbf{l}}(x)}$ for some $x \in C_{\mathbf{k}}(\delta _1)'$. Then by our assumptions, $\Pi_{\mathbf{k}, \mathbf{l}}^{-1}\{y\} \cap \partial Y'' \neq \emptyset$. So $\{y \in \mathbb{R}^2: U \cap \Pi_{\mathbf{k}, \mathbf{l}}^{-1}\{y\} \neq \emptyset\} \subset \Pi_{\mathbf{k},\mathbf{l}}(\partial Y'')$, thus we have
\begin{align*}
	\mathrm{Vol}(U) \leq 3\delta _1 \cdot \mathrm{Area}( \Pi_{\mathbf{k}, \mathbf{l}}(\partial Y'') ) \leq 3\delta _1\cdot  \mathrm{Area}(\partial Y'').
\end{align*}
This proves (\ref{samedi}).
} 
Next, because of Lemma \ref{ybis},
$$3\delta_1 \Area(\partial Y'' ) \leq 10\delta_1 m(g),$$
so together with the above estimate (\ref{cloak}), we would obtain
 $$
 (1- 2m(g) )\delta _1^3 \leq 10\delta_1 m(g),
 $$
which is a contradiction since $0<m(g)\ll 1$.
 }

{Since from the paragraph above, there is a point $x_0' \in \mathbf{C}_{\mathbf{k}}(\delta _1)' $ such that $T_{\mathbf{k}, \mathbf{l}}(x_0') \in \mathbf{C}_{\mathbf{l}}(\delta _1)'$ and $[x_0'T_{\mathbf{k},\mathbf{l}}(x_0')] \subset Y''$, we find}
 \begin{align*}
	 \hat{d}_{\mathrm{Eucl}, Y'' }(x_0, y_0) & \leq \hat{d}_{\mathrm{Eucl}, Y''}(x_0, x_0') + \Length_{\mathrm{Eucl}}([x_0' T_{\mathbf{k}, \mathbf{l}}(x_0')] ) + \hat{d}_{\mathrm{Eucl}, Y''}(T_{\mathbf{k}, \mathbf{l}}(x_0'), y_0)\\
						     &\leq 5 \delta _1 + d_{\mathrm{Eucl}}(x_0', T_{\mathbf{k},\mathbf{l}}(x_0') ) + 5 \delta _1\\
						     & \leq d_{\mathrm{Eucl}}(x_0, y_0) + 20 \delta _1.
 \end{align*}
\end{proof}

The previous lemma implies that any two points of $Y'$ can be joined by  a curve inside $Y''$. The fact that $Y''$ itself is path connected is contained in the following proposition:

\begin{prop}\label{Y'-GH}
For any basepoint $q \in Y''$ and any $D>0$,
	$$
	d_{pGH}( (Y'' \cap B_{\mathrm{Eucl}}(q, D), \hat{d}_{\mathrm{Eucl},Y''},q), (Y''\cap B_{\mathrm{Eucl}}(q,D) , d_{\mathrm{Eucl}},q) ) \leq \Psi (m(g)).
	$$ 
	In particular, $Y''$ is path connected.
\end{prop}

\begin{proof}

	By Lemma \ref{ybis}, $Y''$ lies in the $6\delta_0$-neighborhood of $Y'$ inside $(Y'',\hat{d}_{\mathrm{Eucl},Y''})$. This clearly implies for any $q\in Y'$:
	$$
d_{pGH}( (Y'\cap B_{\mathrm{Eucl}}(q,D), \hat{d}_{\mathrm{Eucl},Y''},q), (Y''\cap B_{\mathrm{Eucl}}(q,D), \hat{d}_{\mathrm{Eucl},Y''},q) ) \leq \Psi (m(g)).
$$

Similarly, since $d_{\mathrm{Eucl}} \leq \hat{d}_{\mathrm{Eucl}, Y''}$, $Y'' $ lies in the $6\delta_0 $-neighborhood of $Y'$ in terms of $d_{\mathrm{Eucl}}$ and for any $q\in Y'$:
$$
d_{pGH}( (Y'\cap B_{\mathrm{Eucl}}(q,D), d_{\mathrm{Eucl}},q), (Y''\cap B_{\mathrm{Eucl}}(q,D), d_{\mathrm{Eucl}},q) ) \leq \Psi (m(g)).
$$

Together with Lemma \ref{ajout} and the triangle inequality, we have the conclusion for in fact any basepoint $q\in Y''$ (using again that $Y''$ lies in the $6\delta_0$-neighborhood of $Y'$ inside $(Y'',\hat{d}_{\mathrm{Eucl},Y''})$).

\end{proof}

Next we can construct a subregion in $E \subset M_{ext}$ by pulling back the subregion constructed above  through the diffeomorphism $\mathbf{u}$.
Set
$$E'':= \mathbf{u}^{-1}(Y'').$$
For any $p \in E'' $ and $D>0$, denote by $\hat{B}_{g,E''}(p, D)$ the geodesic ball in $(E'', \hat{d}_{g, E'' })$, that is,
$$
\hat{B}_{g,E''}(p,D):= \{x \in E'' : \hat{d}_{g, E'' }(p,x) \leq D\} .
$$ 
Similarly, denote by $\hat{B}_{\mathrm{Eucl},Y'' }(q, D)$ the geodesic ball in $(Y'' , \hat{d}_{\mathrm{Eucl},Y'' }).$

\begin{lem}\label{Y'-B-B-hat}
For any basepoint $q \in Y''$ and any $D>0$,
	$$
	d_{pGH}( (Y'' \cap B_{\mathrm{Eucl}}(q, D), \hat{d}_{\mathrm{Eucl},Y''},q), ( \hat{B}_{\mathrm{Eucl},Y''}(q,D) , \hat{d}_{\mathrm{Eucl}, Y'' },q) ) \leq \Psi (m(g)).
	$$ 
\end{lem}
\begin{proof}
	From Lemma \ref{ybis} and (\ref{jung}) in the proof of Lemma \ref{ajout}, for any $q,x\in Y''$, 
	\begin{equation}\label{jung2}
	d_{\mathrm{Eucl}}(q,x) \leq \hat{d}_{\mathrm{Eucl},Y'' }(q, x) \leq d_{\mathrm{Eucl}}(q,x) + \Psi (m(g)),
	\end{equation} 
	so $$\hat{B}_{\mathrm{Eucl},Y''}(q,D) \subset Y'' \cap  B_{\mathrm{Eucl}}(q,D) \subset \hat{B}_{\mathrm{Eucl}, Y''}(q,D + \Psi (m(g) )).$$ 
\end{proof}

\begin{lem}\label{E'-Y'-GH}
	For any basepoint $p \in E''$ and any $D>0$,
	$$
	d_{pGH}( (\hat{B}_{g,E''}(p,D), \hat{d}_{g,E''}, p), (\hat{B}_{\mathrm{Eucl},Y''}(\mathbf{u}(p),D), \hat{d}_{\mathrm{Eucl},Y''}, \mathbf{u}(p) ) ) \leq \Psi (m(g)|D ).
	$$ 
\end{lem}
\begin{proof}
	It is enough to show that $\mathbf{u}$ gives the desired GH approximation. Under the diffeomorphism $\mathbf{u}$, if we still denote the metric $(\mathbf{u}^{-1})^*g$ by $g$, then 
	$$
	(E'' , \hat{d}_{g,E'' }, p) = (Y'' , \hat{d}_{g,Y'' }, \mathbf{u}(p)).
	$$ 
Since $|\Jac \mathbf{u}-\Id| \leq \epsilon' $, we have
$$
\hat{d}_{g, Y'' }(x_1,x_2) \leq (1+ \epsilon') \hat{d}_{\mathrm{Eucl}, Y'' }(x_1,x_2) \leq (1+\epsilon' )^2 \hat{d}_{g, Y''}(x_1,x_2).
$$ 
Note that we have taken $\epsilon = \delta _0$. So for any fixed $D>0$, if $x_1, x_2 \in \hat{B}_{\mathrm{Eucl}, Y'' }(\mathbf{u}(p), D)$, 
$$
| \hat{d}_{g, Y'' }(x_1, x_2)- \hat{d}_{\mathrm{Eucl}, Y'' }(x_1, x_2)| \leq \Psi (m(g)|D ).
$$ 
\end{proof}

From Proposition \ref{Y'-GH}, Lemma \ref{Y'-B-B-hat} and Lemma \ref{E'-Y'-GH}, we immediately have the following.

\begin{lem}\label{E'-Y'}
	For any $p \in E'' $ and $D>0$,
$$
d_{pGH}( (\hat{B}_{g,E''}(p,D), \hat{d}_{g,E''}, p),  (Y''\cap B_{\mathrm{Eucl}}(\mathbf{u}(p),D) , d_{\mathrm{Eucl}}, \mathbf{u}(p) ) \leq \Psi (m(g)|D).
$$ 

\end{lem}

To compare those metric spaces to the Euclidean $3$-space $(\mathbb{R}^3, g_{\mathrm{Eucl}})$, we need the following lemma, which is a corollary of the fact that $\Area(\partial Y'')\leq \Psi (m(g) )$.
\begin{lem} \label{11111}
	For any $q \in Y'' $ and $D>0$,
	$$
	d_{pGH}((Y'' \cap B_{\mathrm{Eucl}}(q,D), d_{\mathrm{Eucl}}, q), (B_{\mathrm{Eucl}}(0, D), d_{\mathrm{Eucl}}, 0) ) \leq \Psi (m(g)).
	$$ 
\end{lem}
\begin{proof}
	Under a translation diffeomorphism, we can assume $q=0$. By (\ref{jung2}), it suffices to show that $B_{\mathrm{Eucl}}(q, D)$ lies in a $\Psi (m(g) ) $-neighborhood of $Y'' $. If that were not the case, there would be a $\mu >0$, independent of $m(g)$, such that for all small enough  $0< m(g) \ll 1$, there exists a $x \in B_{\mathrm{Eucl}}(q, D)$ with $B_{\mathrm{Eucl}}(x, \mu ) \cap Y'' = \emptyset$. 
But from the isoperimetric inequality, we should have 
	$$
	\Vol(\mathbb{R}^3 \setminus Y'' )\leq C \Area(\partial Y'')^{\frac{3}{2}} \leq \Psi (m(g) ),
	$$ 
	which would imply that
	$$
	\omega _3 \mu ^3=\Vol(B_{\mathrm{Eucl}}(x,\mu ) ) \leq \Psi (m(g) ),
	$$ 
	a contradiction.
\end{proof}

Summarizing above arguments, we have proved the following. Recall that $\xi $ is a fixed function defined at the begining of this section.

\begin{prop}\label{main-prop-GH}
Assume $(M, g)$ is a complete, asymptotically flat $3$-manifold,  with nonnegative scalar curvature. Suppose that an end of $(M,g)$ has mass $0< m(g) \ll 1$. Then, there exists a connected closed region $\mathcal{E} \subset M$ containing this end, with smooth boundary, such that 
	$$
	\Area(\partial \mathcal{E}) \leq \frac{m(g)^2}{\xi (m(g) )},
	$$ 
	and there is a harmonic diffeomorphism $\mathbf{u}: \mathcal{E}\to \mathcal{Y}$ with $\mathcal{Y}:= \mathbf{u}(\mathcal{E}) \subset \mathbb{R}^3$ such that, on $\mathcal{E}$, the Jacobian matrix satisfies 
	$$
	|\Jac \mathbf{u}-\Id| \leq \Psi (m(g) ).
	$$ 
	Moreover,  for any basepoint $p \in \mathcal{E}$, any $D>0$,  
	$$
	d_{pGH}( ( \hat{B}_{g,\mathcal{E}}(p, D), \hat{d}_{g, \mathcal{E}}, p), (B_{\mathrm{Eucl}}(0, D), d_{\mathrm{Eucl}}, 0) ) \leq \Psi (m(g)|D),
	$$ 
and $\Phi _{\mathbf{u}(p)} \circ \mathbf{u}$ gives a $\Psi (m(g)|D)$-pGH approximation, where $\Phi _{\mathbf{u}(p)}$ is the translation diffeomorphism of $\mathbb{R}^3$ mapping $\mathbf{u}(p)$ to $0$.
\end{prop}
\begin{proof}
With the same notations as above, we take $\mathcal{E}:= E''$ and $\mathcal{Y}:= Y''$. Notice that by Lemma \ref{ybis}, by the fact that $|\Jac \mathbf{u}-\Id| \leq \epsilon' $ (see (\ref{abu})), and by our choice of $\delta _0$ and $\xi _0$, when $0< m(g) \ll 1$,
$$
\Area(\partial \mathcal{E}) \leq 2\frac{m(g)^2}{\delta _0^5}= 2\frac{m(g)^2}{\xi _0^5(m(g) )}\leq \frac{m(g)^2}{\xi (m(g) )}.
$$ 
The rest of the statement follows from Lemma \ref{E'-Y'} and Lemma \ref{11111}.
\end{proof}
\begin{proof}[Proof of Theorem \ref{main thm}]
	Assume $(M_i, g_i)$ is a sequence of complete, asymptotically flat $3$-manifolds, with nonnegative scalar curvature. Suppose that the mass of an end of $M_i$ converges to zero: $m(g_i)\to 0$. Assume  that $\xi $ is any fixed continuous function as in the statement of Theorem \ref{main thm}. For all large $i$, Proposition \ref{main-prop-GH} gives the existence of a region $\mathcal{E}_i$ with compact boundary,  containing the given end of $M_i$, which satisfies
	$$
	\Area_{g_i}(\partial \mathcal{E}_i) \leq \frac{m(g_i)^2}{\xi (m(g_i) )},
	$$ 
	and a harmonic diffeomorphism $\mathbf{u}_i: \mathcal{E}_i\to \mathcal{Y}_i \subset \mathbb{R}^3$ with $\mathcal{Y}_i= \mathbf{u}_i(\mathcal{E}_i)$. 

	By Proposition \ref{main-prop-GH}, for any basepoint $p_i \in \mathcal{E}_i$, any $D>0$, up to a translation diffeomorphism of $\mathbb{R}^3$, we can assume $\mathbf{u}_i(p_i)=0$, and then $\mathbf{u}_i$ is an $\Psi (m(g_i)|D )$-pGH approximation, and as $i\to \infty$,
	$$
	d_{pGH}( (\hat{B}_{g_i, \mathcal{E}_i}(p_i,D), \hat{d}_{g_i, \mathcal{E}_i}, p_i), (B_{\mathrm{Eucl}}(0, D), d_{\mathrm{Eucl}}, 0) ) \leq \Psi (m(g_i)|D) \to 0.
	$$ 
	In other words,
	$$
	(\mathcal{E}_i, \hat{d}_{g_i, \mathcal{E}_i}, p_i) \to (\mathbb{R}^3, d_{\mathrm{Eucl}}, 0)
	$$ 
	in the pointed Gromov-Hausdorff topology. 

We claim that $	(\mathcal{E}_i, \hat{d}_{g_i, \mathcal{E}_i}, p_i) \to (\mathbb{R}^3, d_{\mathrm{Eucl}}, 0)$ also in the pointed measured Gromov-Hausdorff topology. Since the Hausdorff measure induced by $\hat{d}_{g_i, \mathcal{E}_i}$ is the same as $\dvol_{g_i}$, it suffices to show that for a.e. $D>0$,
\begin{equation}\label{weak conv}
(\mathbf{u}_i)_{\sharp}( \dvol_{g_i}|_{\hat{B}_{g_i, \mathcal{E}_i}(p_i, D)}) \text{ weakly converges to } \dvol_{\mathrm{Eucl}}|_{B(0, D)} \text{ as } i\to \infty.
\end{equation}
By construction and the isoperimetric inequality, $$\Vol(\mathbb{R}^3 \setminus \mathcal{Y}_i) \leq \Psi (m(g_i) ),$$ and so $(\mathcal{Y}_i \cap B_{\mathrm{Eucl}}(0,D), \dvol_{\mathrm{Eucl}})$ converges weakly to $(B_{\mathrm{Eucl}}(0, D), \dvol_{\mathrm{Eucl}}) $. 
Since we have (by abuse of notations):
$$|\Jac \mathbf{u}_i-\Id| \leq \Psi (m(g_i) ),$$ it is now simple to check (\ref{weak conv}) using Lemma \ref{Y'-B-B-hat} and Lemma \ref{E'-Y'-GH}.

We finish the proof by defining  $Z_i:= M_i\setminus \mathcal{E}_i$.

\end{proof}

{
\section{An application to the Bartnik capacity} \label{cb}
We review some definitions concerning the notion of Bartnik capacity, see \cite[Section 9]{HI01}. Given an open Riemannian $3$-manifold $\Omega$ whose metric closure is compact, it is called admissible if there is an extension $(M,g)$ of $\Omega$, such that $(M,g)$ is a complete, connected, asymptotically flat $3$-manifold with one end, having nonnegative scalar curvature, and possessing a minimal compact boundary (which is possibly empty), but no other closed minimal surface in its interior. Such an $(M,g)$ is called an admissible extension. The Bartnik capacity of an admissible $\Omega$ is defined as
$$c_B(\Omega):= \inf \{m(g):\text{ $(M,g)$ is an admissible extension of $\Omega$}\}.$$}
Our convention that $\Omega$ is an open domain follows \cite{HI01,Bartnik02}. 
The next theorem answers Bartnik's strict positivity conjecture \cite[last paragraph of page 2346]{Bartnik89}, and  improves an earlier result of Huisken-Ilmanen \cite[Positivity Property 9.1]{HI01}:
\begin{thm} \label{positive prop}
If $\Omega$ is admissible, then $c_B(\Omega)>0$, unless there is a Riemannian isometric embedding of $\Omega$ into the Euclidean 3-space $\mathbb{R}^3$. 
\end{thm}

\begin{proof}

Let $g_0$ denote the Riemannian metric on $\Omega$.
We can assume that the open domain $\Omega$ is connected without loss of generality.
Let us assume that  $c_B(\Omega)=0$.
Let $(M_i,g_i)$ be a sequence of admissible extensions of $\Omega$, whose masses $m(g_i)$ converge to $0$. 
We use the notations of the previous section and the notation 
$\Psi_i:=\Psi(i^{-1})$ as explained in Section \ref{notations}.

To simplify our task, note that by Huisken-Ilmanen's work \cite[Positivity Property 9.1]{HI01}, we already have: 
\begin{prop}\label{prep}
	$(\Omega ,g_0)$ is locally flat, namely its sectional curvature is zero. 
\end{prop}

By Proposition \ref{main-prop-GH}, there exists a connected closed subset $E_i \subset M_i$ so that
$E_i$ contains the end of $M_i$,  
\begin{equation}\label{cvers0}
\mathrm{Area}(\partial E_i) \to 0
\end{equation}
and for any basepoint $p_i\in E_i$, 
\begin{align}\label{mGH-conv}
	(E_i, \hat{d}_{g_i, E_i}, p_i) \to (\mathbb{R}^3, d_{\mathrm{Eucl}}, 0)
\end{align}
in the pointed measured Gromov-Hausdorff topology. 
Moreover, the measured  Gromov-Hausdorff approximations are given by a sequence of harmonic maps 
\begin{align}\label{u_i}
\mathbf{u}_i: E_i \to \mathbb{R}^3,
\end{align}
depending on $p_i$ and  satisfying the following:
\begin{itemize}
	\item[(1)] $\mathbf{u}_i$ is diffeomorphism onto its image and sends $p_i$ to $0$;
	\item[(2)] the Jacobian matrix of $\mathbf{u}_i$ satisfies $\lim_{i\to \infty} | \mathrm{Jac} \mathbf{u}_i - \mathrm{Id} | =0$;
	\item[(3)] for any $D>0$, for any $x, y \in \hat{B}_{g_i, E_i}(p_i, D)$, 
		$$|\hat{d}_{g_i, E_i}(x,y) - d_{\mathrm{Eucl}}( \mathbf{u}_i(x) , \mathbf{u}_i(y)) | \leq \Psi _i,$$ 
		and 
		$$(\mathbf{u}_i)_{\sharp}(\mathrm{dvol}_{g_i}|_{\hat{B}_{g_i, E_i}(p_i, D)}) \to \mathrm{dvol}_{{\mathrm{Eucl}}}|_{B_{\mathrm{Eucl}}(0,D)} \text{ weakly as measures.}$$
		Here $\Psi_i$ depends on $D$.
\end{itemize}

We will need the following lemma:
\begin{lem}\label{full-vol}
	$\lim_{i\to \infty} \mathrm{Vol}_{g_0}(\Omega \setminus E_i) =0. $
\end{lem}
\begin{proof}[Proof of Lemma \ref{full-vol}]

We can assume that $\partial E_i \neq \varnothing$, otherwise the above equality is clearly satisfied. 
If the lemma is not true, then by the isoperimetric inequality, $\lim_{i\to \infty}\Vol(\Omega\cap E_i)=0.$ Volumes and areas in $\Omega$ are with respect to $g_0$. Fix a large $i$. 
We define the following minimization problem in $(M_i,g_i)$: consider
 $$\alpha:=\inf_{\{\Sigma_t\}_{t\in [0,1]}} \Area(\Sigma_1)$$
 where the infimum is taken over all smooth isotopies $\{\Sigma_t\}_{t\in [0,1]}$ starting at the closed surface $\Sigma_0:=\partial E_i$ and such that for all $t\in [0,1]$, $\Area(\Sigma_t) \leq 2\Area(\partial E_i)$.  
 Two cases could occur.
 \begin{itemize}
 \item Either $\alpha>0$ and, by the work of Meeks-Simon-Yau \cite[Theorem 1]{MSY82}, we obtain a closed minimal surface $S$ of area at most $2\Area(\partial E_i)$ in $(M_i,g_i)$. Since $(M_i,g_i)$ contains no closed minimal surface in its interior, $S$ is contained in the boundary $\partial M_i$. In fact, the statement of \cite[Theorem 1]{MSY82} ensures the following: for any $\epsilon>0$, a family of isotopic surfaces $\{\Sigma_t\}_{t\in [0,1]}$ (with the previous properties) can be chosen so that $\Sigma_1$ is the union of two subsurfaces $\Sigma_{1,1}, \Sigma_{1,2}$ so that $\Sigma_{1,1}$ is inside the $\epsilon$-neighborhood of $S \subset \partial M_i$ and $\Sigma_{1,2}$ has area at most $\epsilon$.
 \item Or $\alpha=0$,  which means that for any $\epsilon>0$,  $\{\Sigma_t\}_{t\in [0,1]}$ (with the previous properties) can be chosen so that $\Sigma_1$ has area at most $\epsilon$.
 \end{itemize}
In both cases, for each $t\in [0,1]$, we can look at the connected region $E_{i,t}\subset M_i$ bounded by $\Sigma_t$ which lies on the same side as $E_i$, in the sense that  $E_{i,t}$ contains the end of $M_i$. Then at $t=0$,  $\Vol(\Omega\cap E_{i,0}) = \Vol(\Omega\cap E_i)$ is small by our assumption, but at $t=1$, $\Vol(\Omega\cap E_{i,1})$ is close to $\Vol(\Omega)$ thanks to the two cases described just above and the general fact that any small area embedded closed surface in a given complete noncompact 3-manifold with bounded geometry (like $M_i$) bounds a small volume region with compact closure, see for instance \cite[Lemma 1]{MSY82}.

 Therefore, there is a time $t_0\in [0,1]$ such that
$$\Vol(\Omega\cap E_{i,t_0}) = \frac{1}{2}\Vol(\Omega)$$
so by the isoperimetric inequality, $\Area(\Omega\cap \partial E_{i,t_0})$ is uniformly lower bounded away from $0$. This is a contradiction with the fact that by construction, for all $t\in [0,1]$ we have $\Area(\Omega\cap  \partial  E_{i,t})\leq \Area(\Sigma_t) \leq 2\Area(\partial E_i)$,  but by (\ref{cvers0})  this last quantity converges to $0$. This proves the lemma.

\end{proof}

For any sufficiently small $\epsilon >0$, consider 
$$\Omega_{\epsilon }:= \{ x \in \Omega; \quad d_{g_0}(x, \partial \Omega ) \geq \epsilon  \} .$$ Then by Proposition \ref{prep}, there exists a positive number 
$$r_\epsilon < \frac{1}{2}\epsilon$$ (independent of $i$) so that any $g_0$-geodesic open $2r_\epsilon $-ball $(B_{g_0}(x,2 r_\epsilon), g_0)$ with center $x \in \Omega _{\epsilon }$ is isometric to a Euclidean $2r_\epsilon $-ball $(B_{\mathrm{Eucl}}(0,2 r_\epsilon), g_{\mathrm{Eucl}})$.

\begin{lem}\label{local-GH}
	For any $x \in \Omega _{3 \epsilon }$, the inclusion map
	\begin{align}
		\iota _i: (E_i \cap B_{g_0}(x,r_\epsilon), \hat{d}_{g_i, E_i} ) \to (B_{g_0}(x,r_\epsilon), d_{g_0})
	\end{align}
	is a $\Psi _i$-GH approximation, and  $(\iota _{i})_{ \sharp} (\mathrm{dvol}_{g_0}|_{E_i \cap B_{g_0}(x,r_\epsilon)})$ weakly converges to  $\mathrm{dvol}_{g_0}|_{B_{g_0}(x,r_\epsilon)}  $. 
	Here, $\Psi_i$ depends on $r_\epsilon$.
\end{lem}
\begin{proof}

	By Lemma \ref{full-vol}, we can choose $x_i \in E_i \cap B_{g_0}(x,r_\epsilon)$ so that $d_{g_0}(x_i, x) \to 0$. Then, by applying (\ref{mGH-conv}) with respect to the basepoints $x_i$, we have:
	$$(\hat{B}_{g_i, E_i}(x_i, 2 r_\epsilon ), \hat{d}_{g_i, E_i} ) \to (B_{\mathrm{Eucl}}(0,2 r_\epsilon ) , d_{\mathrm{Eucl}})$$
	in the measured GH-topology. Here, as in the previous section, $\hat{d}_{g_i, E_i}$ is the induced length metric on $E_i$, and  $\hat{B}_{g_i, E_i}(.,.)$ denotes a corresponding geodesic ball in $E_i$. In particular, for any geodesic ball $\hat{B}_{g_i, E_i}(y_i, r) \subset \hat{B}_{g_i, E_i}(x_i, 2 r_\epsilon ) $, we have the volume convergence 
	\begin{align}\label{local-vol-conv}
		\mathrm{Vol}(\hat{B}_{g_i, E_i}(y_i, r) ) \to \mathrm{Vol}(B_{\mathrm{Eucl}}(0,r) ) = \omega _3 r^3.
	\end{align}

	The following properties suffice to show the lemma:
\begin{itemize}
	\item $B_{g_0}(x,r_\epsilon)$ is contained in the $\Psi_i$-neighborhood of  $E_i \cap B_{g_0}(x,r_\epsilon)$ inside $(\Omega,g_0)$. Otherwise, for some positive constant  $\sigma >0$ independent of $i$, and some $y_i \in B_{g_0}(x,r_\epsilon)$, we have $B_{g_0}(y_i,\sigma) \subset B_{g_0}(x,r_\epsilon) \setminus E_i$. But then 
		\begin{align*}
\omega _3 \sigma ^3 = \mathrm{Vol}(B_{g_0}(y_i,\sigma)) \leq \mathrm{Vol}(B_{g_0}(x,r_\epsilon)\setminus E_i) < \mathrm{Vol}(\Omega \setminus E_i), 			
 		\end{align*}
		which contradicts Lemma \ref{full-vol}.
	
	\item For any $y, z \in E_i \cap B_{g_0}(x,r_\epsilon)$, 
	$$|\hat{d}_{g_i, E_i}(y,z) - d_{g_0}(y,z) | \leq \Psi _i.$$ 
	Indeed, let $\gamma $ be a minimal $\hat{d}_{g_i,E_i}$-geodesic between $y$ and $z$ inside $E_i\cap B_{g_0}(x,r_\epsilon)$. If $\gamma \subset B_{g_0}(x,2r_\epsilon)$, then 
	$$d_{g_0}(y,z) \leq \mathrm{Length}_{g_0}(\gamma ) = \hat{d}_{g_i, E_i}(y,z).$$ Otherwise $\gamma$ is not contained in $B_{g_0}(x,2r_\epsilon)$, and 
	$$d_{g_0}(y,z) \leq 2r_\epsilon\leq \mathrm{Length}_{g_i}(\gamma ) = \hat{d}_{g_i, E_i}(y,z).$$ 
	In both cases, the inequality $ d_{g_0}(y,z) \leq \hat{d}_{g_i, E_i}(y,z)$ follows.
	
Setting $r:= d_{g_0}(y, z)$, we have from the above 
	$$\hat{B}_{g_i, E_i}(y, r ) \subset B_{g_0}(y, r).$$
	 Let us assume that $\hat{B}_{g_i, E_i}(w, \sigma ) \subset B_{g_0}(y, r) \setminus \hat{B}_{g_i, E_i}(y, r )$ for some point $w$ and some $\sigma >0$, then 
		\begin{align*}
			(1-\Psi _i) \omega _3 \sigma ^3 &\leq \mathrm{Vol}( \hat{B}_{g_i, E_i}(w, \sigma ) )\\
			&\leq \mathrm{Vol}(B_{g_0}(y,r) ) - \mathrm{Vol}(\hat{B}_{g_i, E_i}(y,  r) ) \\
			&\leq \Psi _i \omega _3 r^3,
		\end{align*}
where the last inequality comes from the volume convergence (\ref{local-vol-conv}). Hence, we have $\sigma \leq \Psi _i r$. So in fact, there exists $z' \in \hat{B}_{g_i, E_i}(y,  r)$ so that $\hat{d}_{g_i, E_i}(z, z') \leq \Psi _i r$, which implies that 
\begin{align*}
	\hat{d}_{g_i, E_i}(y,z) &\leq \hat{d}_{g_i, E_i}(y, z') + \hat{d}_{g_i, E_i}(z', z)\\
				&\leq r + \Psi _i r\\
				&\leq (1+ \Psi _i) d_{g_0}(y,z).
\end{align*}

	\item $\mathrm{dvol}_{g_0}|_{E_i \cap B_{g_0}(x,r_\epsilon)} \to \mathrm{dvol}_{g_0}|_{B_{g_0}(x,r_\epsilon)}$ weakly. This is a corollary of Lemma \ref{full-vol}.
\end{itemize}
\end{proof}

Next, we construct a limit of  some maps $\mathbf{u}_i$ from $\Omega_{3\epsilon}$ to $\mathbb{R}^3$, which we will show is an isometric embedding. 
Recall that $\Omega$ is supposed to be connected. Fix a basepoint $o\in \Omega$. 
By Lemma \ref{full-vol}, there is a sequence of points $o_i\in E_i$ converging to $o$ in $(\Omega,g_0)$.  
Consider the map given by (\ref{u_i}): 
$$\mathbf{u}_i: (E_i,\hat{d}_{g_i, E_i}) \to \mathbb{R}^3$$
which depends on $o_i$ and sends $o_i$ to $0$. 
By Lemma \ref{local-GH} and using a finite covering of $\Omega_{\epsilon}$ by balls of the form $B_{g_0}(x,r_\epsilon)$, $E_i \cap \Omega _{3 \epsilon } \subset \hat{B}_{g_i, E_i}(o_i, D)$ for some  $D>0$ independent of $i$. Now that  $D$ is fixed, the quantities $\Psi_i$ in  Property (3) of $\mathbf{u}_i$ after (\ref{mGH-conv}) are fixed too.

Importantly, by Property (3) of $\mathbf{u}_i$ after (\ref{mGH-conv}) and Lemma \ref{local-GH}, for any $a,b\in E_i\cap \Omega_{3\epsilon}$ at  $g_0$-distance less than $r_\epsilon$, 
\begin{equation}\label{controll}
d_{g_0}(a,b)-\Psi_i\leq 
d_{\mathrm{Eucl}}(\mathbf{u}_i(a), \mathbf{u}_i(b)) \leq d_{g_0}(a,b)+\Psi_i.
\end{equation}
By a standard Arzel\`a-Ascoli type argument and Lemma \ref{full-vol}, after taking a subsequence,  we can extract a uniform limit of $\mathbf{u}_i$ which is a Lipschitz map:
$$\mathbf{u}_{\infty}: \Omega_{3\epsilon} \to \mathbb{R}^3.$$ 
Moreover, we have the following weak convergence of measures:
$$(\mathbf{u}_{i})_{\sharp} (\mathrm{dvol}_{g_0}|_{E_i\cap \Omega_{3\epsilon}})\to (\mathbf{u}_{\infty})_{ \sharp} \mathrm{dvol}_{g_0}.$$

\begin{lem} \label{omage3 isom}
	$\mathbf{u}_\infty: \Omega_{3 \epsilon } \to \mathbb{R}^3$ is a Riemannian isometric embedding.
\end{lem}
\begin{proof}
	This lemma follows from the following:
\begin{itemize}
	\item  $\mathbf{u}_\infty$ is a local isometry: this  follows from (\ref{controll}) and the construction of $\mathbf{u}_\infty$.

	\item $\mathbf{u}_\infty$ is injective. If that was not true, say $\mathbf{u}_\infty(x) = \mathbf{u}_\infty(y)$ for some $x \neq y \in \Omega _{3 \epsilon }$, then by the local isometry property, there exist small balls $B_{g_0}(x,\sigma), B_{g_0}(y,\sigma)$ so that $B_{g_0}(x,\sigma) \cap B_{g_0}(y,\sigma) = \emptyset$, and $\mathbf{u}_\infty(B_{g_0}(x,\sigma)) = \mathbf{u}_\infty(B_{g_0}(y,\sigma) )=: B_0$. Let $x_i , y_i \in E_i $ and $x_i \to x, y_i\to y$, then both  $\mathbf{u}_i(B_{g_0}(x_i,\sigma) \cap E_i )$ and  $\mathbf{u}_i(B_{g_0}(y_i,\sigma) \cap E_i)$ converge to $B_0$ in the measured GH-topology. This implies that $\mathbf{u}_i(B_\sigma (x_i) \cap E_i) \cap \mathbf{u}_i(B_\sigma (y_i) \cap E_i) \neq \emptyset $, which contradicts with the fact that $\mathbf{u}_i: E_i \to \mathbb{R}^3$ is an injective map (see Property (1) of $\mathbf{u}_i$ listed after (\ref{mGH-conv})).
\end{itemize}
\end{proof}

By Lemma \ref{omage3 isom}, $\Omega_{3\epsilon}$ isometrically embeds inside $\mathbb{R}^3$.
Taking $\epsilon \to 0$, we conclude that $\Omega $ is also isometrically embedded inside $\mathbb{R}^3$. This concludes the proof of Theorem \ref{positive prop}.

\end{proof}

\begin{rmk}  In the proof of Theorem \ref{positive prop}, we use Section \ref{Section 4} crucially to construct a limit map $\mathbf{u}_\infty$ using an Arzel\`a-Ascoli type argument\footnote{We thank a reviewer for catching a mistake in an earlier, different version of this proof.}. 

Section 1.2 in \cite{Song23} applies to maps which are almost Riemannian isometries outside of a region with small volume and small boundary area. It may provide an alternative way of constructing a limit map with measure theoretic methods, and thus  concluding the proof of Theorem \ref{positive prop} without using Section \ref{Section 4}.
	
	Proposition \ref{prep} is not used in an essential way in the proof of Theorem \ref{positive prop} and could be replaced by a slightly more tedious argument.
	 
	\end{rmk}

\vspace{1em}

{
\begin{rmk} \label{aj}
One could try to strengthen Bartnik's strict positivity conjecture in various ways. First, there are other more general ways to define admissible extensions of $\Omega$, by using continuous, but possibly non-smooth, extensions satisfying a mean curvature boundary condition, as studied for instance by Miao \cite{Miao02}, Shi-Tam \cite{ST02}, Anderson-Jauregui \cite{AJ19}. Secondly,  one can ask for an isometric embedding of the \emph{closure} of $\Omega$ into $\mathbb{R}^3$ when $c_B(\Omega)=0$. In \cite[Theorem 1.2]{AJ19}, Anderson-Jauregui constructed counterexamples to such a stronger version of the conjecture (with respect to one of the more general definitions of admissible extensions). 
\end{rmk}
}

\bibliographystyle{alpha}
\bibliography{./math}

\end{document}